\documentclass[letter, 11pt]{article}

\usepackage{natbib}
\usepackage{amssymb,amsmath,xcolor,graphicx,xspace,colortbl,rotating} %
\usepackage[raggedrightboxes]{ragged2e} 
\usepackage{subfigure}
\usepackage{appendix}  
\usepackage{bm} 
\usepackage{cancel} 
\usepackage{boxedminipage}  
\usepackage{color}  
\usepackage{endnotes}  
\usepackage{ragged2e}  
\usepackage[onehalfspacing]{setspace}  
\usepackage{tabulary}  
\usepackage{textcomp}  
\usepackage{varioref}  
\usepackage{graphicx}
\usepackage{caption}
\graphicspath{ {./images/} }
\usepackage[normalem]{ulem}
 
\usepackage[margin=1in]{geometry}
\newtheorem {theorem}{Theorem}

\newtheorem {corollary}{Corollary}

\newtheorem {definition}{Definition}

\newtheorem {lemma}{Lemma}

\newtheorem {proposition}{Proposition}
\newtheorem {remark}{Remark}

\newenvironment {proof}[1][Proof]{\noindent \textbf {#1.} }{\ \rule {0.5em}{0.5em}}
\newcommand{\R}{\mathbb{R}}

\renewcommand{\alph}{n, [a,b]}
\makeatletter
\let\@fnsymbol\@arabic
\makeatother
\begin{document}

\newcommand{\newbl}[1]{{\color{blue}  #1}}
\newcommand{\ap}[1]{{\color{red} [Andres:  #1]}}

\title{A Characterization of the n-th Degree Bounded Stochastic Dominance}

\author{Bar Light\protect\thanks{Business School and Institute of Operations Research and Analytics, National University of Singapore, Singapore. e-mail: \textsf{barlight@nus.edu.sg} } ~ and  Andres Perlroth\protect\thanks{ Google Research, CA, USA. e-mail: \textsf{perlroth@google.com }}}  
\maketitle

\thispagestyle{empty}

  \noindent \textsc{Abstract}:
\begin{quote}

We provide a novel characterization of the $n$-th degree bounded stochastic dominance (BSD) order, linking it to the risk tolerance of decision-makers and providing a decision-theoretic foundation for these stochastic orders. Our results reveal that BSD reflects specific risk preferences through the choice of the interval $[a,b]$, by characterizing it in terms of utility functions with globally bounded Arrow--Pratt risk aversion or that satisfy an $n$-convexity condition. They also highlight limitations of BSD, including its dependence on the chosen support interval and the resulting peculiar risk aversion behavior of decision-makers included in the generator of BSD. To partially address this issue, we use our characterization to separate two roles that are combined in BSD: the largest payoff in the lotteries and the upper endpoint of the interval that determines the Arrow--Pratt lower bound. We then introduce a related lower-partial-moment order that provides a clean trade-off between expected value and downside-risk protection.
Using our characterization, we present comparative statics results for decision-making under uncertainty with globally bounded risk aversion measures and savings decisions under globally bounded prudence measures, and derive inequalities for 
$n$-convex functions.

\end{quote}



\newpage

\section{Introduction}

Stochastic orders are essential tools in fields such as actuarial sciences and decision-making under uncertainty, where they are used to compare random variables. An important class of stochastic orders that has received significant attention from both practitioners and researchers uses the lower partial moments (LPM) of the random variables under consideration to determine the ordering among them \citep{shaked2007stochastic}.  These stochastic orders enjoy well-known desired theoretical properties (e.g., \cite{bawa1975optimal}, \cite{fishburn1976continua}, \cite{fishburn1980stochastic}) and computational properties (e.g., \cite{dentcheva2003optimization}, \cite{dentcheva2016augmented}, \cite{chen2011tight}, \cite{post2017portfolio}, \cite{chen2018stability}, \cite{peng2024data}).

The LPM at a threshold \(c\) quantifies a random variable's downside risk, focusing only on values below \(c\).
 Formally, the $n$-th LPM of a random variable with a distribution function $F$ at a point $c$ is given by $LPM_{n,c}(F)= \int_{ -\infty}^{\infty} \max \{c-x,0\}^{n}dF(x)$. Three different  stochastic orders based on the LPM were introduced and studied in the literature. The most common between them is the stochastic dominance order. 
For a positive integer $n$, we say that $G$ dominates $F$ in the $(n+1)$-th degree stochastic dominance if $LPM_{n,c}(F) \geq LPM_{n,c}(G)$ for all $c \in \mathbb{R}$ \citep{shaked2007stochastic}.  A second stochastic order that is typically called the LPM risk measure considers the inequality $LPM_{n,c}(F) \geq LPM_{n,c}(G)$ for a particular $c$ \citep{shapiro2021lectures}. Finally, a third stochastic order that is called bounded stochastic dominance considers the inequality $LPM_{n,c}(F) \geq LPM_{n,c}(G)$ for all $c$ in some bounded interval $[a,b]$ \citep{fishburn1976continua}. Intuitively, one way to think about  the choice among these stochastic orders is the decision-maker's risk aversion and their knowledge about downside risk. When risk levels are fully known, the LPM risk measure provides a suitable approach for comparing random variables. In scenarios where only partial knowledge is available, bounded stochastic dominance is a more robust method.  For decision-makers without any prior knowledge, or those who are fully risk-averse and wish to hedge against all levels of downside risk, the stochastic dominance order is most suitable. 
This paper elaborates on these distinctions and provides a novel decision-theoretic characterization of bounded stochastic dominance orders, linking these stochastic orders to specific sets of risk-averse decision-makers.

A key advantage of the $(n+1)$-th degree stochastic dominance orders is their well-known characterization in terms of their generators.\footnote{Recall that a generator for a stochastic order $\succeq$ is a set of utility functions $D$ such that $F \succeq G$ implies that every decision maker in $D$ prefers $F$ to $G$, i.e., $\int u(x) dF(x) \geq \int u(x) dG(x)$ for all $u$ in $D$. The generator for the $(n+1)$-th degree stochastic dominance consists of all risk-averse decision makers (i.e., those with a concave and increasing utility function) for $n=1$. For $n=2$, the generator includes all risk-averse and prudent decision makers, that is, those with a concave and increasing utility function and a convex marginal utility function. For higher orders, it includes decision makers who are increasingly risk-averse (see \cite{shaked2007stochastic}). }
 This characterization provides a strong decision-theoretic motivation for employing the $(n+1)$-th degree stochastic dominance orders in applications.
One limitation of these stochastic orders, however, is their strength; they may require too much risk aversion, and hence, fail to compare many random variables of interest. In contrast, the LPM risk measure compares any two random variables, although choosing the threshold $c$ can be somewhat arbitrary, which is a notable drawback.

The bounded stochastic dominance orders serve as a middle ground between the $(n+1)$-th degree stochastic dominance and the LPM risk measure.  Generally, the generator for  $(n+1)$-th degree bounded stochastic dominance orders  is different from the generator of standard $(n+1)$-th degree stochastic dominance orders. This fact was noted in \cite{fishburn1976continua} and \cite{fishburn1980stochastic}. However, a generator for the bounded stochastic dominance order that has a decision-theoretic motivation  is not yet known.\footnote{We also note that in practice,  due to computational considerations, to determine if \(G\) dominates \(F\) in the \((n+1)\)-th degree stochastic dominance, one has to restrict the dominance relation to some bounded domain \([a, b]\) (see \cite{shapiro2021lectures}).}

Our main result characterizes the bounded stochastic dominance orders by providing  a generator for these orders that has a decision-theoretic interpretation and is closely related to the Arrow-Pratt risk aversion measure and the ``$n$-convexity" of the utility functions that belong to the generator.\footnote{A non-negative function $u$ is $n$-convex if $u^{1/n}$ is a convex function.} We show that the generator for the $n+1$-th degree bounded stochastic dominance consists of the utility functions $u$ satisfying (i) the {\em changing derivative property}: the $k$-derivative satisfies $(-1)^k u^{(k)}\leq 0$ for $k=1,\ldots, n+1$ and (ii) $u$ is $n,[a,b]$ convex, i.e.,  $(u(b)-u(x))^{1/n}$ is a convex function on $[a,b]$. Alternatively, (ii) can be replaced by the following condition (ii') the Arrow-Pratt measure of risk aversion $-u^{(2)}(x)/u^{(1)}(x)$ is bounded below by $(n-1)/(b-x)$ for $x \in (a,b)$. The changing derivative property is a well-desired property from a decision-theoretic  point of view and it  characterizes the standard $(n+1)$-degree stochastic dominance order as we mentioned above. Condition (ii) and (ii') provide restrictions on the risk tolerance of the decision makers. Intuitively, these restrictions imply that decision makers that belong to the generator exhibit some amount of risk aversion.  We show that this generator is essentially the maximal generator \citep{muller1997stochastic} of these stochastic orders.  

The core contribution of our characterization is its ability to establish a clear and explicit link between bounded stochastic dominance (BSD) orders defined by lower partial moments and the risk preferences of decision-makers. It was previously known that the choice of the interval $[a,b]$ in the definition of BSD alters the stochastic order under consideration; our characterization extends this understanding by connecting the interval $[a,b]$ to classical measures of risk aversion, such as the Arrow-Pratt risk aversion measure. Specifically, we demonstrate that BSD is linked to utility functions with globally bounded Arrow-Pratt risk aversion measures from below, thereby shedding new light on the class of utility functions for which BSD serves as an appropriate analytical tool. In other words, our results show that the bounded stochastic dominance criteria can be viewed as representing the preferences of a specific class of decision-makers with well-defined and explicitly parameterized risk attitudes. 

The same characterization also highlights limitations of BSD. One of the key advantages of classic stochastic dominance (SD) orders is their perceived ``robustness.'' They require relatively few and agreeable assumptions about utility functions, such as risk-aversion for second-order SD, prudence for third-order SD, temperance for fourth-order SD, and are therefore often regarded as robust and broadly applicable criteria. However, our results show how bounded stochastic dominance (BSD) orders depend on the choice of the interval $[a,b]$ for orders beyond the second, making BSD tied to specific assumptions about risk preferences that go beyond concavity. A particular concern stemming from our characterization is the implication of infinite risk aversion at the upper bound $b$. Specifically, utilities in the generator of the BSD order exhibit infinite Arrow-Pratt risk aversion as $x$ approaches $b$. From a purely economic standpoint, this assumption is potentially unrealistic.  Without a clear justification for the interval $[a,b]$ grounded in economic or behavioral considerations, BSD loses the decision-theoretic appeal of SD as the set of utility functions it represents may be seen as ``unnatural'' for many standard economic applications. 

To address this issue, we use our characterization to separate two roles that are combined in BSD: the largest payoff in the lotteries and the upper endpoint of the interval that determines the Arrow-Pratt lower bound. Specifically, for lotteries supported on $[a,b]$, we consider utility classes defined by a larger endpoint $c>b$. This yields a related lower-partial-moment order that keeps risk aversion finite on the lottery support and provides a clean trade-off between expected value and downside-risk protection. In the third- and fourth-degree cases, we derive simple conditions showing when stronger downside-risk protection can compensate for lower expected value, and how this trade-off becomes more stringent as the endpoint $c$ moves farther from the support.

As a byproduct of our characterization, we establish an equivalence between \(n\)-convexity and functions whose Arrow–Pratt measures of risk aversion are bounded from below. In Section \ref{sec:n-convex}, we leverage this equivalence to derive novel inequalities for \(n\)-convex functions, including Jensen-type inequalities.
In Sections \ref{Sec:pratt} and \ref{Sec:Prudence}, we use our results to show how decision makers with globally bounded from below Arrow-Pratt risk aversion measure rank non-trivial lotteries and how decision makers with globally bounded from below prudence measure choose how much to save under different uncertain future incomes. Our results hold for global bounds on the Arrow-Pratt risk aversion measure and prudence measure as opposed to the typical local analysis using the Arrow-Pratt risk aversion measure or the prudence measure (e.g., \cite{kimball1990precautionary}).


There is an extensive literature on stochastic orders, particularly focusing on stochastic dominance orders (for a survey, see \cite{muller2002comparison} and \cite{shaked2007stochastic}) and stochastic orders that are generated by various sets of functions. Higher stochastic  dominance orders (see \cite{whitmore1970third},  \cite{rolski1976order}, \cite{ekern1980increasing}, and \cite{denuit1998s}) restrict the sign of higher derivatives of functions and are based on the concept of $n$-monotonicity \citep{williamson1955multiply, rolski1976order}. \cite{leshno2002preferred} and \cite{tsetlin2015generalized} impose certain decision-theoretic motivated constraints on the derivatives of these functions. \cite{vickson1977stochastic} and \cite{post2014linear} assume that functions belong to the Decreasing Absolute Risk Aversion (DARA) class. \cite{post2016standard} adds further curvature conditions on higher derivatives. Fractional degree stochastic dominance is explored by \cite{muller2016between}, \cite{huang2020fractional}, \cite{mao2022characterizing}, and \cite{azmoodeh2023multi}.
\cite{meyer1977choice} studies utility functions with upper and lower bounds on the Arrow-Pratt absolute risk-aversion measure, thus studying the stochastic order generated by the set $AP_{n,[a,b]}$ described in Section \ref{Section: Main}. \cite{light2021family} investigate the class of $n,[a,b]$-convex functions and the corresponding stochastic order generated by the set of $n,[a,b]$-concave functions $LP_{n,[a,b]}$ defined in Section \ref{Section: Main}. The main difference in our approach, in contrast to much of the existing literature, lies in our approach to defining the stochastic orders we study. While previous literature generally define a stochastic order using a set of utility functions grounded in decision theory and subsequently attempt to identify simple integral conditions sufficient to characterize the stochastic order, our paper adopts the inverse process. We start with a well recognized and studied simple integral condition, the comparison of random variables through lower partial moments and then proceed to characterize the set of utility functions, that has a clear  decision-theoretic interpretation, generated by this stochastic order. 

Stochastic orders have been extensively used to compare random variables and to derive comparative statics results in economics, e.g., in information design \citep{ivanov2021optimal}, the measurement of discrimination and inequality \citep{le2012stochastic, zheng2021stochastic}, Bayesian games \citep{mekonnen2022bayesian}, and the analysis of stationary equilibria \citep{light2020uniqueness}, just to name a few. As discussed above, we contribute to this literature in our applications section by presenting novel comparisons of decision outcomes under different lotteries in classical settings such as expected utility theory and consumption–savings problems.

\section{Characterization of Bounded Stochastic Dominance} \label{Section: Main}
In this section we provide our main result. We first introduce some notations. 

Let $C^{n}([a,b])$ be the set of all $n$ times continuously differentiable functions $u :[a ,b] \rightarrow \mathbb{R}$. For $k\geq 1$, we denote by $u^{(k)}$ the $k$th derivative of a function $u$ and for $k=0$ we define $u^{(0)}:=u$.\footnote{As usual, the derivatives at the extreme points $u^{(k)}(a)$ and $u^{(k)}(b)$ are defined by taking the left-side and right-side limits, respectively, see Definition 5.1 \cite{rudin1964principles}).}  For a non-negative integer $n$ and real numbers $a<b$ we define the following sets of utility functions that have the changing derivative property:
\begin{equation*}{U}_{\alph} : =\{u \in C^{n+1}([a,b]): \: (-1)^{k}u^{(k)}(x) \leq 0 \text { } \forall x \in [a,b]  \text{ } \; \forall k =1 ,\ldots  ,n+1 \}.
\end{equation*}

The sets  of utility functions ${U}_{\alph}$ play an important rule in the theory of stochastic orders (for example, see \cite{hadar1969rules}, \cite{whitmore1970third}, \cite{menezes1980increasing}, and \cite{ekern1980increasing}). Note that the set $U_{1,[a,b]}$ consists of all risk averse decision makers and the set $U_{2,[a,b]}$ consists of all risk averse decision makers that are also prudent \citep{kimball1990precautionary}. Stochastic orders that are generated by the sets ${U}_{\alph}$ have received a significant  attention in the literature and are typically called the $(n+1)$-th degree stochastic dominance orders  \citep{muller2002comparison,shaked2007stochastic}. The utility functions  that belong to the sets $U_{n,[a,b]}$ have desired well-known theoretical properties, and hence, stochastic orders that are generated by these sets are well motivated from a decision theory point of view. In addition, there is a known connection between these stochastic orders to the lower partial moments (LPM) of the random variables under consideration (see Proposition \ref{prop:known} below). This connection is fundamental for practical approaches that are used to determine if a random variable dominates another random variable under these stochastic orders   (see \cite{post2017portfolio} and \cite{fang2022optimal}).  We now define formally the $(n+1)$-th degree stochastic dominance.

\begin{definition}\label{Def: Stochastic dominance}
Let $F$ and $G$ be two distribution functions on $[a,b]$ and let $n$ be a positive integer. We say that $G$ dominates $F$ in $(n+1)$-th degree stochastic dominance order and write $G \succeq_{n+1} F$, if and only if for all $c \in \mathbb{R}$, we have
 \begin{equation} \label{equation:LPM}
     \int_{-\infty}^{\infty} \max\{c-x,0\}^{n} dF(x) \ge \int_{-\infty}^{\infty} \max\{c-x,0\}^{n} dG(x)\;.
      \end{equation} 
\end{definition}

For a proof of the following characterization of the $(n+1)$-th degree stochastic dominance see, for example,  \cite{shaked2007stochastic} Section 4.A.7 (the moment based conditions that are typically imposed to state Proposition \ref{prop:known} are redundant under inequality (\ref{equation:LPM}), e.g., \cite{fishburn1980stochastic}.) 
\begin{proposition}  \label{prop:known}
Let $F$ and $G$ be two distributions over $[a,b]$ and let $n$ be a positive integer. Then $G \succeq_{n+1} F$  if and only if 
\begin{equation*}
\int _{a}^{b} u(x) dG(x) \geq  \int _{a}^{b} u(x) dF(x)
\end{equation*}
for all $u \in U_{n,[a,b]}$. 
\end{proposition}

As we discussed in the introduction, the $(n+1)$-th degree stochastic dominance can be too strong in the sense that many random variables cannot be compared by these stochastic orders. An alternative approach is to use a``restricted" $(n+1)$-th degree stochastic dominance where inequality (\ref{equation:LPM}) is required to hold only for  $c \in [a,b]$ for some bounded interval $[a,b]$.   These stochastic orders were introduced by \cite{fishburn1976continua} and are called bounded stochastic dominance orders.  We now define formally  bounded stochastic dominance (BSD). 
\begin{definition}
Let $F$ and $G$ be two distribution functions on $[a,b]$ and let $n$ be a positive integer. We say that $G$ dominates $F$ in  the $(n+1)$-th degree bounded  stochastic dominance order and write $G \succeq_{n+1,[a,b]-BSD} F$, if and only if for all $c \in [a,b]$, we have
 \begin{equation} \label{equation:bddLPM}
     \int_{a}^{b} \max\{c-x,0\}^{n} dF(x) \ge \int_a^b \max\{c-x,0\}^{n} dG(x)\;.
      \end{equation} 
\end{definition}

Clearly stochastic dominance implies bounded stochastic dominance. Notice that, given Proposition \ref{prop:known},  a generator of the $(n+1)$-th degree bounded stochastic dominance can be given by a set of decisions makers $U_{n,[a,b]} \cap W_{n,[a,b]}$ for some set  $W_{n,[a,b]}$.  Our main result provides the set $W_{n,[a,b]}$ for each $n$ (see Theorem \ref{Main:Thm}). We show that these sets contain ``very" risk averse decision makers in the sense that their utility functions satisfy curvature conditions that are stronger than concavity. These curvature conditions are related to the Arrow-Pratt risk aversion measure and the $n$-convexity of the utility functions. We now introduce these sets. 

We first discuss the set of $n,[a,b]$-concave functions. We say that a function is $n,[a,b]$-concave if it belongs to the set $LP_{\alph}$ where 
\begin{equation*}{LP}_{\alph} : =\{u :[a,b] \rightarrow \mathbb{R}: \left (u(b)-u(x) \right )^{1/n} \text{ is convex and decreasing on } [a,b]  \},
\end{equation*}
i.e., the function $u(b)-u(x)$ is decreasing and $n$-convex on $[a,b]$. 
For a twice continuously differentiable concave and strictly increasing  function $u$ with the normalization $u(b)=0$,  we have $u \in LP_{\alph}$ if and only if 
\begin{equation*}\frac{\partial \ln (u^{(1)}(x))}{ \partial \ln (-u(x))} =\frac{u(x)u^{(2)}(x)}{\left (u^{(1)}(x)\right )^{2}} \geq \frac{n  -1}{n } \label{Ineq: p-convex}
\end{equation*} 
for all $x \in (a ,b)$, i.e.,  the elasticity of the marginal utility function with respect to the utility function is bounded below by $(n -1)/n $. The elasticity of $u^{(1)}$ with respect to $u$ is a natural measure of the concavity of $u$ and has an economic interpretation  in terms of the trade-off between risk and reward. \cite{light2021family} study the properties of the functions that belong to the set $LP_{\alph}$ and the stochastic orders that are generated by these functions. 

The second set we introduce is a set of decision makers that have a bounded from below Arrow-Pratt measure of risk aversion. Recall that the Arrow-Pratt measure of risk aversion is given by $-u^{(2)}/u^{(1)} = -\partial \ln(u^{(1)})$. 
\begin{equation*}{AP}_{\alph} : =\{u \in C^{2}([a,b]): u^{(1)}(x) \geq 0, \:  u^{(2)}(x) \leq 0, \: (n-1) u^{(1)}(x) + u^{(2)}(x)(b -x) \leq 0  \text{ } \; \forall x \in [a,b]  \}.
\end{equation*}
For a twice continuously differentiable concave and strictly increasing  function $u$  we have $u \in AP_{\alph}$ if and only if \begin{equation*} -\partial \ln(u^{(1)}(x)) = -\frac{u^{(2)}(x)}{u^{(1)}(x)}  \geq \frac{n  -1}{b-x} \label{Ineq: ADc}
\end{equation*} 
for all $x \in (a ,b)$, i.e., the Arrow-Pratt measure of risk aversion is bounded below by $(n-1)/(b-x)$. 
The Arrow-Pratt measure of risk aversion \citep{pratt1964risk} has many well-known desired theoretical properties and is widely used in the decision theory literature to measure the risk tolerance of decision makers. \cite{meyer1977choice} and \cite{meyer1977second} study the stochastic orders that are generated by functions in $AP_{n,[a,b]}$ (see Remark \ref{Remark:1} below).

 One disadvantage of the sets $AP_{\alph}$ and $LP_{\alph}$  is that they are not subsets of  $U_{\alph}$. For example, for $n=2$, the sets consist of risk averse decision makers that might not be averse to downside risk (prudent) which is a desired property in many applications of interest \citep{kimball1990precautionary}.   Another disadvantage  is that the stochastic orders that are generated by these sets are not easy to characterize (see \cite{meyer1977choice} and \cite{light2021family}). On the other hand, the sets  $U_{\alph} \cap LP _{\alph}$ or $U_{\alph} \cap AP _{\alph}$ include risk averse decision makers that exhibit higher orders of risk aversion (e.g., prudence). For example, for $n=2$ the set  $U_{\alph} \cap AP _{\alph}$ consists of all risk averse and prudent decision makers such that their Arrow-Pratt measure of risk aversion at a point $x$ is at least $1/(b-x)$. Our main result shows that the stochastic orders generated by these sets have a simple characterization in terms of lower partial  moments. 
Theorem \ref{Main:Thm} shows that for every positive integer $n$ these two sets are equal, i.e.,   $U_{\alph} \cap LP _{\alph} =  U_{\alph} \cap AP _{\alph}$ and that these sets generate the $(n+1)$-th degree bounded stochastic dominance order. 

The proof of Theorem \ref{Main:Thm} is provided in Section \ref{sec:proof}.

\begin{theorem} \label{Main:Thm}
Let $F$ and $G$ be two distributions over $[a,b]$ for some $a<b<\infty$ and let $n$ be a positive integer. These three conditions are equivalent:

(i) $G \succeq_{n+1,[a,b]-BSD} F$

(ii) We have
\begin{equation} \label{Eq: 1}
\int _{a}^{b} u(x) dG(x) \geq  \int _{a}^{b} u(x) dF(x)
\end{equation}
for all $u \in U_{\alph} \cap LP _{\alph}$. 

(iii) Inequality  (\ref{Eq: 1}) holds for all $u \in U_{\alph} \cap AP _{\alph}$.

 In particular, we have $U_{\alph} \cap AP _{\alph} = U_{\alph} \cap LP _{\alph}$.
\end{theorem}

Under suitable integrability conditions, Theorem \ref{Main:Thm} also holds for the case that $a= -\infty$. On the other hand, when $b=\infty$ the bounded stochastic order reduces to the standard unbounded stochastic dominance order (see Proposition   \ref{prop:known} above).

Theorem \ref{Main:Thm} shows that $U_{\alph} \cap AP _{\alph} = U_{\alph} \cap LP _{\alph}$.  One may wonder whether $AP_{\alph} = LP_{\alph}$. The next proposition shows that this is not the case. The proof is deferred to the Appendix.

 \begin{proposition} \label{Prop:notequal}
Let $a<b$ and $n \geq 2$. We have     $AP_{n,[a,b]} \nsubseteq LP_{n,[a,b]}$.
\end{proposition} 

We end this section with a few remarks regarding our characterization.

\begin{remark}
 Identifying functions in $U_{\alph} \cap AP _{\alph} = U_{\alph} \cap LP _{\alph}$: At first glance, identifying functions \(u \in U_{\alph}\) that also belong to \(AP_{\alph}\) or \(LP_{\alph}\) may seem non-trivial. However, Lemma \ref{Lemma:Taylor} provides  a straightforward method to identify such functions by considering their behavior at the endpoint \(b\) of the interval \([a, b]\) and their Taylor expansions. 
In the proof of Theorem \ref{Main:Thm} we show that  $U_{\alph} \cap LP _{\alph}= U_{\alph} \cap AP _{\alph} = \mathcal{G}_{n,[a,b]}$ where $\mathcal{G}_{n,[a,b]}$ is defined in Equation (\ref{Eq:G_set}) above. Therefore, a function \(u \in U_{\alph}\) can be checked for membership in \(U_{\alph} \cap AP_{\alph}\) simply by verifying that its derivatives at \(b\) satisfy the conditions for \(\mathcal{G}_{n,[a,b]}\). For example, consider the function $v_{\gamma}$ defined in Equation (\ref{Eq: Almost CRRA}) in Section \ref{Sec:pratt} that is closely related to the class of constant relative risk aversion functions. Then it is immediate to verify that it is indeed in $\mathcal{G}_{2,[a,b]}$ so it belongs to $U_{2,[a,b]} \cap AP _{2,[a,b]}$ as claimed in Section \ref{Sec:pratt}.
We now state Lemma \ref{Lemma:Taylor} that outlines a systematic approach to generate functions in \(U_{\alph} \cap AP_{\alph} = U_{\alph} \cap LP_{\alph}\) from any function in \(U_{\alph}\):  
\begin{lemma} \label{Lemma:Taylor}
    Suppose that $u \in U_{n+1,[a,b]}$ for $n \geq 1$. Consider the function $$g(x) = u(x) - u(b) - \sum _{j=1}^{n} \frac {u^{(j)}(b) (x-b)^{j} } { j!}.$$ Then $g \in U_{n+1,[a,b]} \cap AP _{n+1,[a,b]} = U_{n+1,[a,b]} \cap LP _{n+1,[a,b]}$.
\end{lemma}

\end{remark}

\begin{remark} \label{Remark:1}
    
\cite{meyer1977choice} and \cite{meyer1977second} study stochastic orders generated by sets of decision makers with bounded Arrow-Pratt measures of risk aversion. Notably, for the function \(k(x) = -(b-x)^n\) where $n$ is a positive integer, it holds that \(-k^{(2)}(x)/k^{(1)}(x) = (n-1)/(b-x)\). Thus, for any strictly increasing function \(u\), it belongs to \(AP_{n,[a,b]}\) if and only if \(u(x) = \rho(k(x))\) for some strictly increasing and concave function \(\rho\); this is based on the standard characterization of risk aversion measures from \cite{pratt1964risk}. Furthermore, invoking Theorem 3.7 from \cite{schilling2009bernstein}, we have that \(u(x) = \rho(k(x))\) belongs to \(U_{n,[a,b]}\) if \(\rho \) itself belongs to \(U_{n,[a,b]}\). This provides an interesting characterization for the strictly increasing functions $u$ that belong to \(U_{n,[a,b]} \cap AP_{n,[a,b]}\) as those such that \(u(x) = \rho(k(x))\) for some strictly increasing function \(\rho\) in \(U_{n,[a,b]}\).

\end{remark}

\subsection{The Maximal Generator of Bounded Stochastic Dominance} \label{sec:maximalgen}

Theorem \ref{Main:Thm} provides a generator for the bounded stochastic dominance orders. In this section we discuss the maximal generator of these stochastic orders. Formally, define $F \succeq _{\mathfrak{F}}G$ if 
\begin{equation}\label{Ineq:ineqGandF} \int _{a}^{b}u(x)dF(x) \geq \int _{a}^{b}u(x)dG(x)
\end{equation}
for all $u \in \mathfrak{F}$ where $\mathfrak{F}$ is a subset of the set of all bounded utility functions on $[a,b]$. 
The maximal generator $R_{\mathfrak{F}}$ of the stochastic order $ \succeq _{\mathfrak{F}}$ is the set of all functions $u$ with the property that $F \succeq _{\mathfrak{F}}G$ implies inequality (\ref{Ineq:ineqGandF}).

 When using a stochastic order in decision theory contexts, it is useful to characterize the maximal generator of the stochastic order. If the maximal generator is not known, it is not clear which  decision makers are under consideration when deciding if a random variable dominates another random variable. 
\cite{muller1997stochastic} characterizes the properties of maximal generators. We use the results in \cite{muller1997stochastic} to identify the maximal generator of the bounded stochastic dominance orders. The following result follows from Corollary 3.8 in \cite{muller1997stochastic}.  

\begin{proposition}\label{Prop: maximal generator}
   Suppose that $\mathfrak{F}$ is a convex cone containing the constant functions and closed under pointwise convergence. Then  $R_{\mathfrak{F}}=\mathfrak{F}$.
\end{proposition}

It is easy to see that $U_{n,[a,b]}$ is a convex cone that contains the constant functions. The closure of $U_{n,[a,b]}$ denoted by $cl(U_{n,[a,b]})$ can be written by divided differences \citep{denuit1999stochastic}.  \cite{light2021family} show that $LP_{n,[a,b]}$ is a convex cone containing the constant functions and closed under pointwise convergence. Hence, the following Corollary follows. 
\begin{corollary} \label{Corr:maximal generator}
    The maximal generator of the $(n+1)$-th degree bounded stochastic dominance order $\succeq _{n+1,[a,b]-BSD}$ is $ cl(U_{n,[a,b]} \cap LP_{n,[a,b]}) $. 
\end{corollary} 

Interestingly, in the proof of Theorem \ref{Main:Thm} we show that  
$U_{\alph} \cap LP _{\alph}= U_{\alph} \cap AP _{\alph} = \mathcal{G}_{n,[a,b]}$ 
where for an integer $n \geq 1$ and $a<b$ we denote 
\begin{equation} \label{Eq:G_set} \mathcal{G}_{n,[a,b]} := \{u \in U_{\alph}: u^{(k)}(b) = 0 \text{ } \forall k=1,\ldots,n-1  \}.
\end{equation}
Hence, the maximal generator of the $(n+1)$-th degree bounded stochastic dominance order is also given by $cl (\mathcal{G}_{n,[a,b]})$.

\begin{remark}
   Given Proposition \ref{Prop: maximal generator}, we may consider the maximal generator of the \((n+1)\)-th bounded stochastic dominance order as represented by the set:
\[
\mathcal{W}_{n} := \{ u:[a,b] \rightarrow \mathbb{R} \mid \exists \mu \in \mathcal{M}_{[a,b]}, \text{  } u(x) = -\int \max (c-x)^{n} \, \mu(dc) \},
\]
where \(\mathcal{M}_{[a,b]}\) denotes the set of non-negative Borel measures on \([a,b]\) (also see \cite{fishburn1976continua}). The set \(\mathcal{W}_{n}\) is a closed convex cone, and along with constant functions, it generates the \((n+1)\)-th bounded stochastic dominance order. One way to view this paper's main result is that it characterizes this generator through utility functions, providing a decision-theoretic interpretation of the \((n+1)\)-th degree bounded stochastic dominance order.
\end{remark}

\section{Expected Value and Downside Risk Trade-offs under LPM Dominance}
One aspect of bounded stochastic dominance, evident in the set $AP_{n,[a,b]}$ for $n\geq 2$, is the condition that risk aversion approaches infinity for values near the upper bound of the lottery's support, $b$. This condition appears counterintuitive in standard decision theory because it requires marginal utility to be $0$ at the endpoint of the support, implying absolute satiation at that point. However, in typical economic applications, a decision maker's absolute satiation point $c$ may be much larger than the maximum possible realization $b$. This can limit the applicability of bounded stochastic dominance orders, as the generated set of utility functions may not be economically sensible in certain contexts.

This infinite risk aversion at the upper bound stems from the inherent structure of the bounded stochastic dominance order. Specifically, lower partial moments (LPMs) of the form $\max\{t-x,0\}^n$, for $t\in[a,b]$, focus on downside risk within the interval $[a,b]$. Here, the boundary at $b$ represents a threshold beyond which outcomes are either impossible or excluded from consideration. Because this formulation inherently ignores values beyond $b$, the marginal utility at $b$ is forced to $0$, which leads to infinite risk aversion at that point. 

To address this, in this section we consider scenarios where the decision maker's utility is defined over an extended interval $[a,c]$, with $c>b$. This decouples the lottery's maximum payoff from the decision maker's satiation point. In this case, the Arrow-Pratt lower bound becomes $\frac{n-1}{c-x}$, which remains finite over the entire relevant support $[a,b]$ of the random variables. Consequently, the set of utility functions under consideration becomes more reasonable for such lotteries. We now characterize these orders in terms of conditions on lower partial moments.

The next result makes precise how extending the utility domain from $[a,b]$ to $[a,c]$, while keeping the lotteries supported on $[a,b]$, captures a trade-off between expected value and downside protection. When $c=b$, we recover the original bounded stochastic dominance order on $[a,b]$; as $c$ tends to infinity we get the standard stochastic dominance order. The intermediate case $b<c<\infty$  separates the lottery's maximal payoff from the decision maker's satiation point. 
For $n=2$, the characterization is especially transparent. The conditions on $[a,b]$ are the usual second-order LPM inequalities, while the additional condition at $t=c$ can be written as
\[
LPM_{2,b}(F)-LPM_{2,b}(G)
+
2(c-b)\left(LPM_{1,b}(F)-LPM_{1,b}(G)\right)\geq 0.
\]
The term $LPM_{1,b}(F)-LPM_{1,b}(G)=\mu_G-\mu_F$ measures the mean advantage of $G$ over $F$. Thus, a lottery $G$ with a lower mean can still be preferred by all decision makers in the relevant generator, provided that its second-order downside protection at $b$ is sufficiently stronger; the distance $c-b$ determines how much mean loss must be compensated in this way.
 For $n=3$, the same mechanism remains, but an additional curvature restriction appears: when $G$'s mean advantage is accompanied by a second-order LPM disadvantage at $b$, its third-order LPM gain must be large enough to compensate for that loss. 

\begin{theorem} \label{Thm:Extended_BSD_n23}
Let $F$ and $G$ be two distribution functions supported on $[a,b]$, and let $c>b$. The following two characterizations hold.

\begin{enumerate}
    \item[(i)] For $n=2$, we have 
    $\int_a^b u(x)dG(x)\geq \int_a^b u(x)dF(x)$ for every 
    $u\in U_{2,[a,c]}\cap LP_{2,[a,c]}=U_{2,[a,c]}\cap AP_{2,[a,c]}$
    if and only if
    \begin{equation*}
    LPM_{2,t}(F)\geq LPM_{2,t}(G)
    \quad \text{for all } t\in [a,b]\cup\{c\}.
    \end{equation*}

    \item[(ii)] For $n=3$, we have 
    $\int_a^b u(x)dG(x)\geq \int_a^b u(x)dF(x)$ for every 
    $u\in U_{3,[a,c]}\cap LP_{3,[a,c]}=U_{3,[a,c]}\cap AP_{3,[a,c]}$
    if and only if the following conditions hold:
    \begin{enumerate}
        \item[(a)] 
        \begin{equation*}
        LPM_{3,t}(F)\geq LPM_{3,t}(G)
        \quad \text{for all } t\in [a,b]\cup\{c\}.
        \end{equation*}

        \item[(b)] If 
        \begin{equation*}
        LPM_{1,b}(F)-LPM_{1,b}(G)>0
        \end{equation*}
        and
        \begin{equation*}
        -2(c-b)\left(LPM_{1,b}(F)-LPM_{1,b}(G)\right)
        <LPM_{2,b}(F)-LPM_{2,b}(G)<0,
        \end{equation*}
        then
        \begin{equation*}
        4\left(LPM_{1,b}(F)-LPM_{1,b}(G)\right)
        \left(LPM_{3,b}(F)-LPM_{3,b}(G)\right)
        \geq
        3\left(LPM_{2,b}(F)-LPM_{2,b}(G)\right)^2.
        \end{equation*}
    \end{enumerate}
\end{enumerate}
\end{theorem}

\section{Examples and Applications}

In this section we present some examples and applications of Theorem \ref{Main:Thm}.

\subsection {Decision Making and the Arrow-Pratt Risk Aversion Measure} \label{Sec:pratt}

For a decision maker with a utility function $u$, \cite{pratt1964risk} recognized that $-u^{(2)}/u^{(1)}$ is a useful measure for the decision maker's risk aversion strength. The theory provided in \cite{pratt1964risk} focuses on small changes in the  lottery that the decision maker faces and the risk premium that is associated with that lottery.  
Theorem \ref{Main:Thm} provides a connection between the global boundness of the risk aversion measure $-u^{(2)}/u^{(1)}$ to the decision maker's ranking over general lotteries. In particular, if $F \succeq_{n+1,[a,b]-BSD} G$, then $-u^{(2)}(x)/u^{(1)}(x)  \geq (n-1)/(b-x)$ for all $x \in (a,b)$ and $u \in U_{n,[a,b]}$ implies that the decision maker prefers $F$ to $G$.

As an example, consider a simple case where $G$ yields $a$ dollars with probability $\theta ^{n}$ and $b$ dollars with probability $1 - \theta^{n}$, and $F$ yields $\theta a + (1-\theta)b$ dollars with probability $1$ for $\theta \in (0,1)$, $b>a$. Then it is easy to check that $F \succeq_{n+1,[a,b]-BSD} G$. Hence, Theorem \ref{Main:Thm} shows that if a decision maker has the changing derivative property, i.e., $u \in U_{n,[a,b]}$ and a globally bounded Arrow-Pratt risk aversion measure $-u^{(2)}(x)/u^{(1)}(x) \geq (n-1)/(b-x)$ for all $x \in (a,b)$, then the decision maker prefers $F$ to $G$. Note that $G$'s expected value is higher than $F$'s expected value so not every risk averse decision maker prefers $F$ to $G$. Hence, our results provide non-trivial rankings of random variables for decision makers that have a globally bounded from below Arrow-Pratt measure of risk aversion. 

We can also use Theorem \ref{Thm:Extended_BSD_n23} to consider the more general case in which the lotteries are supported on $[a,b]$, but the endpoint that determines the lower bound on the Arrow--Pratt measure is some $c>b$. To see this in the simplest case, take $n=2$. Let $G$ yield $a$ with probability $\theta^2$ and $b$ with probability $1-\theta^2$, and let $F$ yield $\hat\theta a+(1-\hat\theta)b$ with probability $1$. By Theorem \ref{Thm:Extended_BSD_n23}, applied with $F$ and $G$ interchanged, 
$F$ is preferred to $G$ by every decision maker in 
$U_{2,[a,c]}\cap AP_{2,[a,c]}$ if and only if
$
LPM_{2,t}(G)\geq LPM_{2,t}(F)
\quad \text{for all } t\in [a,b]\cup\{c\}.
$ 
For $t\in [a,b]$, these inequalities hold whenever $\hat\theta\leq \theta$. 
The additional condition at $t=c$ is
\[
\theta^2(c-a)^2+(1-\theta^2)(c-b)^2
\geq
\left(c-\hat\theta a-(1-\hat\theta)b\right)^2.
\]
Equivalently,
\[
\hat\theta\leq \widetilde{\theta}(c):=
\frac{
\sqrt{\theta^2(c-a)^2+(1-\theta^2)(c-b)^2}-(c-b)
}{b-a}.
\]
Therefore, whenever
$
\theta^2<\hat\theta\leq \widetilde{\theta}(c) 
$ 
the lottery $G$ has a higher expected value than $F$, but every decision maker in 
$U_{2,[a,c]}\cap AP_{2,[a,c]}$ prefers $F$ to $G$.
In the non-trivial case where $G$ has a higher expected value than $F$ we have $\theta^2<\hat\theta$ so the relevant trade-off is determined by combining this condition with the requirement $\hat\theta\leq \widetilde{\theta}(c)$.

When $c=b$, we obtain $\widetilde{\theta}(b)=\theta$. Hence the dominance condition reduces to the original condition $\hat\theta\leq \theta$ discussed above.  
At the other extreme, $\lim_{c\to\infty}\widetilde{\theta}(c)=\theta^2$. Hence, as $c$ tends to infinity, the interval $(\theta^2,\widetilde{\theta}(c)]$ shrinks to the empty set. This recovers the intuition from the usual (unbounded) stochastic-dominance criterion: if $G$ has a strictly higher expected value than $F$, then $F$ cannot dominate $G$. 
For $b<c<\infty$, the upper bound $\widetilde{\theta}(c)$ lies strictly between $\theta^2$ and $\theta$. Thus, there is still a non-empty range of values of $\hat\theta$ for which $G$ has a higher expected value but $F$ is preferred by all decision makers in the generator. As $c$ increases, this range becomes smaller. Intuitively, increasing $c$ means that stronger downside-risk protection is needed to compensate for the lower expected value of $F$.

Using Lemma \ref{Lemma:Taylor} we can easily generate functions that belong $AP_{n,[a,b]} \cap U_{n,[a,b]}$ from functions that belong to $U_{n,[a,b]}$. For example, interestingly, a closely related class of utility functions to the popular class of constant relative risk aversion (CRRA) utility functions $u_{\gamma}(x) = x^{1-\gamma} / (1-\gamma)$ belongs to $U_{2,[a,b]} \cap AP_{2,[a,b]}$. From Lemma \ref{Lemma:Taylor}, the functions
 \begin{equation} \label{Eq: Almost CRRA}
     v_{\gamma}(x)= u_{\gamma}(x) - \frac{x}{b^{\gamma}}
      \end{equation}
for $\gamma > 0$,  $\gamma \neq 1$ and $v_{1}(x)= \log(x) - x/b$ for $\gamma = 1$ belong to $U_{2,[a,b]} \cap AP_{2,[a,b]}$. Decision makers with a utility function $v$  balance risk and expected value. If the expected values of $F$ and $G$ are the same, they behave exactly as CRRA decision makers, but if the expected values are different, then both the difference between the expected values and the behavior of the CRRA decision maker are taken into account when deciding between $F$ and $G$. The magnitude of balancing depends on $b$.

\subsection{Prudence and Savings Decisions} \label{Sec:Prudence}

The impact of future income uncertainty, particularly on savings decisions, has been extensively analyzed in economic research. Commonly, the literature employs stochastic orders such as second order stochastic dominance (SOSD) or higher orders that impose a ranking over expectations of random incomes. It is well known that in a two-period consumption-savings model, if the risk associated with labor income increases in the sense of SOSD, a prudent agent (i.e., an agent with a convex marginal utility function) will increase their savings \citep{sandmo1970effect} (see \cite{nocetti2015robust}, \cite{light2017precautionary}, and   \cite{bommier2018risk}) for various recent extensions of this result).  Additionally, an increase in the expected present value of future labor income tends to reduce current savings.
However, when comparing two different random incomes, where one is riskier but has a higher expected value, the savings behavior of a prudent agent becomes less clear. In this section, we employ our characterization of bounded stochastic dominance orders to analyze such scenarios. We show that comparisons can be made if the agent’s prudence measure is globally bounded from below. We now describe the consumption-savings model. 

Consider an agent who must decide how much to save and how much to consume, with their next period’s income being uncertain. If the agent has an initial wealth of \(x\) and chooses to save an amount \(0 \leq s \leq x\), then the utility from the first period’s consumption is \(u(x-s)\), and the utility from the second period is \(u(Rs + y)\), where \(y\) represents the uncertain income, and \(R\) is the rate of return which is assumed to be deterministic for simplicity. The utility function \(u\) characterizes the agent’s  preferences and is assumed to be strictly increasing, strictly concave, and continuously differentiable.

The agent's objective is to maximize the expected utility:
\begin{equation*}
w(s, F) := u(x - s) + \int_{0}^{\overline{y}} u(Rs + y) dF(y),
\end{equation*}
where \(F\) is the distribution function of the next period's income \(y\), supported on \([0, \overline{y}]\).

Let \(\rho(F) = \ensuremath{\operatorname*{argmax}}_{s \in C(x)} w(s, F)\) denote the optimal savings under the distribution \(F\), where \(C(x) := [0, x]\) represents the range of feasible savings levels given the agent's wealth \(x\).

The prudence of an agent is quantified by the measure \(-u'''/u''\) as discussed in \cite{kimball1990precautionary}. The main result of this section establishes that if the prudence measure is globally bounded from below, then the agent saves more under $F$ than under $G$ when \(G \succeq_{n+1,[a,b]-BSD} F\) where the interval $[a,b]$ represents potential future consumption values. The result extends the classical finding that prudent agents tend to save more under greater uncertainty in labor income which corresponds to \(n=1\) to more complex situations where the agent faces future random incomes that are riskier yet potentially have higher expected values. Importantly, this result is global in the sense that the prudence measure's lower bound applies universally rather than locally, which is different than the typical local analyses studied in previous literature \citep{kimball1990precautionary}.
In particular, the next proposition shows that if the prudence measure is bounded from below $-u'''(z)/u''(z) \geq (n-1)/(Rx+\overline{y}-z)$ and $-u \in U_{n,[0,Rx +\overline{y}]}$ then $ \int_{0}^{Rx +\overline{y}} \max\{c-x,0\}^{n} dF(x) \ge \int_{0}^{Rx +\overline{y}} \max\{c-x,0\}^{n} dG(x)$ for all $c \in [0,Rx +\overline{y}]$ implies that $\rho(F) \geq \rho(G)$. That is, the agent increases savings when moving from $G$ to $F$, where $F$ is riskier and possibly has a higher expected value than $G$, when the prudence measure is bounded from below.

\begin{proposition} \label{prop-consumption}
Let $n$ be a positive integer. Assume that the agent's marginal utility $-u^{(1)} \in U_{n,[0,Rx +\overline{y}]} \cap AP_{n,[0 ,Rx +\overline{y}]}$, i.e., the agent prudence measure is bounded from below. Then  $G  \succeq _{n+1 ,[0 ,Rx +\overline{y}]-BSD} F$ implies that $\rho(F) \geq \rho(G)$.
\end{proposition}

    \begin{remark}
In Proposition \ref{prop-consumption}, the endpoint $Rx+\overline{y}$ is the largest possible second-period consumption, and it is also the endpoint that determines the lower bound on the agent's prudence measure. 
As in Section \ref{Sec:pratt}, one can use Theorem \ref{Thm:Extended_BSD_n23} to separate these two roles. Specifically, one may choose an endpoint $C>Rx+\overline{y}$ and impose the lower bound $-u'''(z)/u''(z)\geq (n-1)/(C-z)$ on the relevant consumption range. This keeps the prudence bound finite on the support of second-period consumption. The corresponding LPM conditions are then obtained by applying Theorem \ref{Thm:Extended_BSD_n23} to the income lotteries with the larger endpoint. 
\end{remark}
\subsection{Relation to $n$-convex Functions and Inequalities}\label{sec:n-convex}

Theorem \ref{Main:Thm} has non-trivial implications for $n$-convex functions.\footnote{The set of $n$-convex functions has been studied in convex geometry (e.g.,  \cite{lovasz1993random}) and in economics (e.g.,  \cite{jensen2017distributional}).} In many applications in economics, operations research, and mathematics a subset of convex functions is studied in order to obtain stronger theoretical results than those that can be proven for convex functions (for example, log-convex functions, strongly convex functions, and $n$-convex functions). Corollary \ref{Cor:convex} establishes a useful characterization for $n$-convex functions. We  provide a connection between the convexity of $f(x)^{1/n}$, i.e.,  the $n$-convexity of $f$, to the convexity of $f(x^{1/n})$. This Corollary can be used to prove interesting inequalities for $n,[a,b]$-concave functions. For example, Corollary \ref{Corr:jensen}  establishes a  Jensen-type inequality for  $n,[a,b]$-concave functions that is tighter than the standard  Jensen's inequality for concave functions. The proofs are deferred to the Appendix.  

\begin{corollary} \label{Cor:convex}
Let $a<b$ and let $n$ be a positive integer. Assume that $f \in U_{n,[a,b]}$. Then $g_{n}(x):=(f(b)-f(x))^{1/n}$ is convex on $[a,b]$ if and only if $k_{n}(x):=-f(b-x^{1/n})$ is convex on $[0,(b-a)^{n}]$.
\end{corollary}

\begin{corollary} \label{Corr:jensen}
Let $a<b$ and let $n$ be a positive integer. Assume that $f \in U_{n,[a,b]}$ and $(f(b)-f(x))^{1/n}$  is convex on $[a,b]$. Then
$$ \mathbb{E}f(X) \leq f \left (b -(\mathbb{E}(b-X)^{n})^{1/n} \right ) \leq f(\mathbb{E}(X))$$
for every random variable $X$ on $[a,b]$.
\end{corollary}

\section{Proof of Theorem \ref{Main:Thm} } \label{sec:proof}
Fix an integer $n \geq 1$ and $a<b$. We first define the following set of functions: 
$$\mathcal{G}_{n,[a,b]} := \{u \in U_{\alph}: u^{(k)}(b) = 0 \text{ } \forall k=1,\ldots,n-1  \}.$$ 

The proof proceeds with the following Lemmas. In Lemma \ref{Lemma: G=AP} we show that $\mathcal{G}_{n,[a,b]} = U_{\alph} \cap AP _{\alph} $. In Lemma \ref{prop:new} we show that $\mathcal{G}_{n,[a,b]} = U_{\alph} \cap LP _{\alph} $, and hence, $U_{\alph} \cap LP _{\alph}= U_{\alph} \cap AP _{\alph}$. Lemma \ref{lemm:aux} is a technical Lemma that is needed in order to prove Lemma \ref{Lemma: LPM iff G}. In Lemma \ref{Lemma: LPM iff G}   we show that $G \succeq_{n+1,[a,b]-BSD} F$ if and only if  $$ \int _{a}^{b}u(x)d(F-G)(x) \leq 0 \text { for all }  u \in \mathcal{G}_{n,[a,b]}.$$
Combining this with Lemma 1 and Lemma 2 completes the proof of Theorem \ref{Main:Thm}.

\begin{lemma} \label{Lemma: G=AP}
        We have $\mathcal{G}_{n,[a,b]} = U_{\alph} \cap AP _{\alph} $. 
\end{lemma}

\begin{proof}
The proof is immediate for $n=1$. Assume that $n \geq 2$. 
Let $u \in \mathcal{G}_{n,[a,b]}$. 
We will show that 
\begin{equation} (n-1) u^{(1)}(x) + u^{(2)}(x)(b -x) \leq 0. \label{Eq: AP}
\end{equation}
for all $x \in [a,b]$. 

Suppose that $n$ is an even number. Note that if $n$ is an even number then the fact that  $u \in \mathcal{G}_{n,[a,b]}$ implies that $u^{(n-1)}$ is convex. Hence, the derivative of $u^{(n-1)}$ is increasing, i.e., $u^{(n)}$ is increasing. Using the fact that $u^{(n-1)}(b) =0$  implies that for all $x \in [a,b]$ we have
\begin{equation}u^{(n-1)}(x) =u^{(n-1)}(b) - \int _{x}^{b}u^{(n)}(t)dt = - \int _{x}^{b}u^{(n)}(t)dt \leq - \int _{x}^{b}u^{(n)}(x)dt = - (b -x)u^{(n)}(x). \label{Jensen: ineq2}
\end{equation}
Similarly, if $n$ is an odd number for all $x \in [a,b]$ we have  
\begin{equation}u^{(n-1)}(x) =u^{(n-1)}(b) - \int _{x}^{b}u^{(n)}(t)dt = - \int _{x}^{b}u^{(n)}(t)dt \geq - \int _{x}^{b}u^{(n)}(x)dt = - (b -x)u^{(n)}(x). \label{Jensen: ineq2.1}
\end{equation}
 Thus, if $n = 2$ then inequality (\ref{Eq: AP}) holds.

To prove that inequality  (\ref{Eq: AP}) holds for $n \geq 3$, define
$$z(x) =   (n-1)u^{(1)}(x)  + u^{(2)}(x)(b -x)  \text { on } [a,b].$$ Note that $z^{(1)}(x) = (n -2)u^{(2)}(x) + u^{(3)}(x)(b -x)$,
and 
\begin{equation*}z^{(k)}(x) = (n-1-k) u^{(k +1)}(x) + u^{(k +2)}(x)(b -x)
\end{equation*}
for all $k=0,\ldots, n-1$. 
Because $u^{(k)}(b) =0$ for all $k=1,\ldots,n-1$, we have $z^{(k)}(b) =0$ for all  $k=0,\ldots,n-2$.  

Suppose first that $n$ is an even number. 
Inequality (\ref{Jensen: ineq2}) yields $z^{(n - 2)}(x) \leq 0$ for all $x \in [a ,b]$. Thus, $z^{(n -3)}$ is a decreasing function. Combining this with the fact that $z^{(n -3)}(b) =0$ implies that $z^{(n -3)}(x) \geq 0$ for all $x \in [a ,b]$. Using the same argument, repeatedly,  it follows that $z^{(j)}(x) \leq 0$ for all $x \in [a ,b]$ and for  all even numbers $0 \leq j \leq n-2$ and $z^{(j)}(x) \geq 0$ for all $x \in [a ,b]$ and for all odd numbers $0 \leq j \leq n-2$. In particular, $z^{(0)}(x) :=z(x) \leq 0$ for all $x \in [a ,b]$.  When $n$ is an odd number,  using inequality (\ref{Jensen: ineq2.1}), an analogous argument to the argument above shows that  $z(x) \leq 0$ for  all $x \in [a ,b]$.  We  conclude that inequality  (\ref{Eq: AP}) holds for every integer $n \geq 1$, i.e., $ u \in U_{\alph} \cap AP _{\alph}$. 

Now assume that $ u \in U_{\alph} \cap AP _{\alph} $. Then $z (x) \leq 0$ on $[a,b]$ which implies that $z(b) = (n-1)u^{(1)}(b) \leq 0 $. The fact that $u^{(1)} \geq 0 $ implies $u^{(1)}(b) = 0$. Thus,  we have $z(b)=0$. 

Assume in contradiction that there exists some $j = 2,\ldots,n-1$ such that $u^{(j)}(b) \neq 0$. Then there exists an integer $2 \leq m \leq n-1$ that satisfies  $u^{(m)}(b) \neq 0$ and $u^{(k)}(b) = 0$ for all  $k=1,\ldots,m-1$. We have  $z^{(k)}(b)=0$ for all $k=0,\ldots,m-2$ and $z^{(m-1)}(b) = u^{(m)}(b)(n-m) < 0$ if $m$ is an even number and $z^{(m-1)}(b) >0$ if $m$ is an odd number. Thus, if $m$ is an even number, then the   continuity of $z^{(m-1)}$ implies that $z^{(m-1)}(x) <0 $ for all $x \in [b-\epsilon,b]$ for some $\epsilon>0$. Combining this with the fact that $z^{(m-2)}(b)=0$ implies that there exists a $\delta >0$ such that $z^{(m-2)}(x) > 0$ for all $x \in (b-\delta,b)$. Using the fact that 
$z^{(k)}(b)=0$ for all $k=0,\ldots,m-2$ and the argument above repeatedly show that $z^{(0)}(x) :=z(x) > 0 $ for some $x$ close to $b$ which is a contradiction to $z (x) \leq 0$ on $[a,b]$. Similarly, if $m$ is an odd number then $z^{(m-1)}(b) > 0$ and $z^{(m-2)}(b)=0$ imply that $z^{(m-2)}(x) <0 $ for all $x \in (b-\epsilon,b)$ for some $\epsilon>0$ which leads to a contradiction from the argument above.      We conclude that $u^{(k)} (b) = 0$ for all $k=1,\ldots,n-1$. 
Thus, $u \in \mathcal{G}_{n,[a,b]}$.

\end{proof}

\begin{lemma}\label{prop:new}
We have $\mathcal{G}_{n,[a,b]} = U_{\alph} \cap LP _{\alph} $.
\end{lemma} 

\begin{proof} 
The proof is by induction on $n$. The base case, $n=1$, trivially holds. 

For the sake of exposition, we divide the proof to 9 steps. Step 1 shows that we can restrict the analysis to a simpler class of functions. Steps 2-4 prove that $\mathcal{G}_{n,[a,b]} \subseteq  U_{n,[a,b]} \cap LP_{n,[a,b]}$. Steps 5-9 prove that $U_{n,[a,b]} \cap LP_{n,[a,b]} \subseteq  \mathcal{G}_{n,[a,b]}$.

{\bf Step 1:} If $\mathcal{G}_{n,[a,b]} =U_{n,[a,b]} \cap   LP_{n,[a,b]}$ for functions $u$ satisfying $|u^{(k)}(x)|>0$ for every $x<b$ and $k=1,\ldots, n$ then $\mathcal{G}_{n,[a,b]} =U_{n,[a,b]} \cap   LP_{n,[a,b]}$. 

Suppose that $\mathcal{G}_{n,[a,b]} =U_{n,[a,b]} \cap   LP_{n,[a,b]}$ for functions $u$ satisfying $|u^{(k)}(x)|>0$ for every $x<b$ and $k=1,\ldots, n$. Let $u\in \mathcal{G}_{n,[a,b]}$ and observe that $-\epsilon(b-x)^{n}\in \mathcal{G}_{n,[a,b]}$ for every $\epsilon > 0$. The set $\mathcal{G}_{n,[a,b]}$ is convex so the function $u_\epsilon(x)=u(x)-\epsilon(b-x)^{n}$ belongs to the set $\mathcal{G}_{n,[a,b]}$. Moreover, $|u_\epsilon^{(k)}(x)|>0$ for $x<b$ and $k=1,\ldots, n$. Thus, we conclude that $u_\epsilon \in U_{n,[a,b]} \cap   LP_{n,[a,b]}$. Because the set $  LP_{n,[a,b]}$ are closed under the pointwise topology, we have that when $\epsilon\to 0$, $u \in LP_{n,[a,b]}$, and hence, $u \in U_{n,[a,b]} \cap LP_{n,[a,b]}$. Because $-\epsilon(b-x)^{n} \in LP_{n,[a,b]}$, a similar argument implies the inclusion  $U_{n,[a,b]} \cap LP_{n,[a,b]} \subseteq \mathcal{G}_{n,[a,b]}$. Hence, $\mathcal{G}_{n,[a,b]} =U_{n,[a,b]} \cap   LP_{n,[a,b]}$. 

\medskip


Let $u \in \mathcal{G}_{n,[a,b]}$. For the next steps we assume that $|u^{(k)}(x)|>0$ for every $x<b$ and $k=1,\ldots, n$. 

We define the function 
$$R(x)=\frac{(u(x)-u(b) )u^{(2)}(x)}{u^{(1)}(x)^2} \text { on } (a,b). $$ 
Because $u^{(1)}(x)>0$ for $x\in (a,b)$, the function $R$ is well-defined and differentiable on $(a,b)$. From the differentiability of $u$ it follows that showing that  $ u \in LP_{n,[a,b]}$ is equivalent to showing that for all $x \in (a,b)$ we have 
\begin{equation} \label{Eq: Lem2: 1}
R(x) =\frac{(u(x)-u(b) )u^{(2)}(x)}{u^{(1)}(x)^2}  \geq \frac{n-1}{n}.
\end{equation}

{\bf Step 2:} We claim that 
\begin{align}\label{eq:new_R}
    R^{(1)} (x) = \frac 1 {u^{(1)}(x)^2}\bigg[&-u^{(1)}(x)u^{(2)}(x)\left(\frac{n}{n-1}R(x)-1 \right ) \nonumber \\  &+ (u(x)-u(b) ) \left(u^{(3)}(x) - \frac{n-2}{n-1} \frac{u^{(2)}(x)^2}{u^{(1)}(x)} \right) \bigg]
\end{align}
for all $x \in (a,b)$.

Indeed, we have 
 \begin{align*}
R^{(1)} (x)&=\frac{1}{u^{(1)}(x)^4}\bigg[u^{(1)}(x)^3 u^{(2)}(x) + (u(x)-u(b))u^{(1)}(x)^2u^{(3)}(x) - 2 (u(x)-u(b))u^{(1)}(x)u^{(2)}(x)^2\bigg]\\
&= \frac{1}{u^{(1)}(x)^2}\bigg[u^{(1)}(x) u^{(2)}(x) + (u(x)-u(b))u^{(3)}(x) - 2 \frac{(u(x)-u(b))u^{(2)}(x)^2}{u^{(1)}(x)}\bigg]\\
&= \frac{1}{u^{(1)}(x)^2}\bigg[ u^{(1)}(x)u^{(2)}(x)\bigg(1-2 \underbrace{\frac{(u(x)-u(b))u^{(2)}(x)}{u^{(1)}(x)^2}}_{R(x)}\bigg) +  (u(x)-u(b))u^{(3)}(x)  \bigg]\\
&=\frac 1 {u^{(1)}(x)^2}\bigg[-u^{(1)}(x)u^{(2)}(x)(2R(x)-1) + (u(x)-u(b))u^{(3)}(x) \bigg] \\
&= \frac 1 {u^{(1)}(x)^2}\bigg[-u^{(1)}(x)u^{(2)}(x)\left(\frac{n}{n-1}R(x)-1\right) - \frac{n-2}{n-1} \underbrace{u^{(1)}(x)u^{(2)}(x) R(x)}_{\frac{(u(x)-u(b))u^{(2)}(x)^2}{u^{(1)}(x)}} + (u(x)-u(b))u^{(3)}(x) \bigg]\\
&= \frac 1 {u^{(1)}(x)^2}\bigg[-u^{(1)}(x)u^{(2)}(x)\left(\frac{n}{n-1}R(x)-1\right) + (u(x)-u(b)) \left(u^{(3)}(x) - \frac{n-2}{n-1} \frac{u^{(2)}(x)^2}{u^{(1)}(x)} \right) \bigg].
\end{align*}

{\bf Step 3:} Let $R(b) := \liminf_{x\to b^-} R(x)$, then we assert that  $R(b)\ge \frac{n-1}{n}$. 

First, observe that $\lim_{x\to b^-}{(u(x)-u(b)) u^{(2)}(x)}=0$ and  $\lim_{x\to b^{-}} u^{(1)}(x)^2=u^{(1)}(b)^2=0$. Hence,
\begin{align}
R(b)&=\liminf_{x\to b^-}\frac{(u(x)-u(b)) u^{(2)}(x)}{u^{(1)}(x)^2} \nonumber \\
&\geq \liminf_{x\to b^-}\frac{u^{(1)}(x) u^{(2)}(x) + (u(x)-u(b))u^{(3)}(x)} {2u^{(1)}(x)u^{(2)}(x)} \nonumber\\
&=\frac 1 2 + \liminf_{x\to b^-}  \frac{(u(x)-u(b))u^{(3)}(x)} {2u^{(1)}(x)u^{(2)}(x)} \nonumber\\
&=\frac 1 2 + \liminf_{x\to b^-}  \frac 1 2  \frac{(u(x)-u(b))u^{(2)}(x)}{u^{(1)}(x)^2} \frac{u^{(1)}(x) u^{(3)}(x)}{u^{(2)}(x)^2} \nonumber\\
&\ge \frac 1 2 +   \frac 1 2\liminf_{x\to b^-}  \frac{(u(x)-u(b))u^{(2)}(x)}{u^{(1)}(x)^2} \liminf_{x\to b^-} \frac{u^{(1)}(x) u^{(3)}(x)}{u^{(2)}(x)^2} \nonumber\\
& = \frac 1 2 +   \frac 1 2 \;R(b)\; \liminf_{x\to b^-} \frac{u^{(1)}(x) u^{(3)}(x)}{u^{(2)}(x)^2}. \label{ineq:s3}
\end{align}
The first inequality follows from the general version of L'Hopital rule. The second inequality follows using that $\liminf _{x \to b^{-} } f(x)g(x) \geq \liminf _{x \to b^{-} } f(x) \liminf _{x \to b^{-} } g(x) $ for any functions $f$ and $g$. 

Second,  note that the function $v(x)=-u^{(1)}(x)$ is in the set $ \mathcal{G}_{n-1,[a,b]}$. The inductive hypothesis implies that $v \in LP_{n-1,[a,b]}$. Then, we have that
$$\frac{(v(x)-v(b)) v^{(2)}(x)}{v^{(1)}(x)^2} \ge \frac{n-2}{n-1}.$$
Using $v(x)=-u^{(1)}(x)$ and $u^{(1)}(b)=0$, we derive that
$$ \frac{u^{(1)}(x) u^{(3)}(x)}{u^{(2)}(x)^2} \ge \frac{n-2}{n-1}.$$
Hence, $\liminf_{x\to b^-} \frac{u^{(1)}(x) u^{(3)}(x)}{u^{(2)}(x)^2} \ge \frac{n-2}{n-1}.$

To conclude this step we combine the last inequality with Inequality~\eqref{ineq:s3} obtaining
$$R(b) \ge \frac 1 2 + \frac 1 2 \; R(b) \; \frac{n-2}{n-1}.$$
This simplifies to $R(b) \ge \frac{n-1}{n}$.

{\bf Step 4:} We claim that $u\in LP_{n,[a,b]}$. That is, for all $x\in (a,b)$ we have that $R(x) \ge \frac{n-1}{n}$. 

Suppose for the sake of  contradiction that a $y \in (a,b)$ exists such that $R(y)<\frac{n-1}{n}$. Let $z=\inf\{w\in [y,b]| R(w)>R(y)\}$. Step 3 shows that $\liminf_{x\to b^-}R(x) \ge \frac{n-1}{n}> R(y)$, which in turn, implied that $z$ is well-defined and $z<b$. Moreover, because $R$ is continuous we have that $R(z)=R(y)<\frac{n-1}{n}$. From the definition of the infimum we have that
\begin{align}\label{con1_new}
\forall \epsilon>0 \; \exists w_\epsilon\in (z,z+\epsilon)\mbox{ such that } R(w_\epsilon)>R(y)=R(z).
\end{align}


As in Step 3, we take use $v(x)=-u^{(1)}(x)$. Note that $v$ is an  increasing function. From the inductive hypothesis to obtain that 
$$v^{(2)}(z) \leq \frac{n-2}{n-1} \frac{v^{(1)}(z)^2}{v(z)-v(b)}$$
Because $v(b)=- u^{(1)}(b)=0$, we conclude that 
$$u^{(3)}(z) - \frac{n-2}{n-1} \frac{u^{(2)}(z)^2}{u^{(1)}(z)} \geq 0.$$ 
Combining the above inequality and $R(z)<\frac{n-1}{n}$ and the facts that $u$ is strictly increasing and strictly concave, we get that 
$$ -u^{(1)}(z)u^{(2)}(z)\left(\frac{n}{n-1}R(z)-1 \right )  + (u(z)-u(b) ) \left(u^{(3)}(z) - \frac{n-2}{n-1} \frac{u^{(2)}(z)^2}{u^{(1)}(z)} \right) < 0$$
From Step 2 we conclude that 
$R^{(1)}(z)<0$. Thus, 
\begin{align}\label{con2_new}
\exists \epsilon>0 \mbox{ such that }\forall w \in (z,z+\epsilon) \text{ we have }  R(w) < R(z).
\end{align}
conditions~\eqref{con1_new} and ~\eqref{con2_new} are mutually impossible. Therefore, we conclude that for every $x\in (a,b)$, $R(x)\ge \frac {n-1}{n}$.

\medskip

We now show that  if $u\in U_{n,[a,b]} \cap LP_{n,[a,b]}$ then $u\in \mathcal{G}_{n,[a,b]}$.




{\bf Step 5:} For $k=1,\ldots,n-2$ we have that $u^{(k)}(b)=0$.

 It can be shown that $ LP_{n,[a,b]} \subseteq  LP_{n-1,[a,b]}$ (see for example \cite{light2021family}). Hence, if $u\in U_{n,[a,b]} \cap LP_{n,[a,b]}$ then $u\in  U_{n,[a,b]} \cap LP_{n-1,[a,b]}\subseteq U_{n-1,[a,b]}\cap  LP_{n-1,[a,b]} $. By the inductive hypothesis we conclude that $u\in \mathcal{G}_{n-1,[a,b]}$. Thus, $u^{(k)}(b)=0$ for $k=1,\ldots,n-2$. 
 
 The next three steps conclude the proof by showing that $u^{(n-1)}(b)=0$.

{\bf Step 6:} For $k=0,\ldots,n-2$, let $$R_{u^{(k)}}(b)= \limsup_{x\to b^-}\frac{(u^{(k)}(x)-u^{(k)}(b)) u^{(k+2)}(x)}{u^{(k+1)}(x)^2}.$$ We assert that $R_{u^{(k)}}(b)\ge 0$ for  $k=0,\ldots,n-2$. 

Because $u$ is increasing and concave we have $(u(x)-u(b))u^{(2)}(x) \ge 0$. Thus, Step 6 holds for $k=0$. For $1 \leq k \leq n-2$, because $u^{(k)}(b)=0$ (Step 5), we get that 
$$(u^{(k)}(x)-u^{(k)}(b)) u^{(k+2)}(x) = u^{(k)}(x)u^{(k+2)}(x)=(-1)^ku^{(k)}(x)\cdot (-1)^{k+2}u^{(k+2)}(x) \geq 0$$
where the last inequality follows because $(-1)^{k} u^{k}(x) < 0$ and  $(-1)^{k+2} u^{k+2}(x) < 0$ for every $x\in (a,b)$. We conclude that $R_{u^{(k)}}(b)\ge 0$. 

{\bf Step 7:} Suppose $n>2$, then for $k=0,\ldots,n-3$ we assert that
\begin{equation}\label{rec:1} R_{u^{(k)}}(b) \le \frac 1 2 \left(1+ R_{u^{(k)}}(b) R_{u^{(k+1)}}(b)\right).
\end{equation}
Step 5 shows that $\lim_{x\to b^-}u^{(k+1)}(x) =u^{(k+1)}(b)=0$ and $\lim_{x\to b^-} u^{(k)}(x)=u^{(k)}(b)$. Now, observe that
\begin{align*}
R_{u^{(k)}}(b)&=\limsup_{x\to b^-}\frac{(u^{(k)}(x)-u^{(k)}(b)) u^{(k+2)}(x)}{u^{(k+1)}(x)^2}\\ &\leq \limsup_{x\to b^-}\frac{u^{(k+1)}(x) u^{(k+2)}(x) + (u^{(k)}(x)-u^{(k)}(b))u^{(k+3)}(x)} {2u^{(k+1)}(x)u^{(k+2)}(x)}\\
&=\frac 1 2 + \limsup_{x\to b^-}  \frac{(u^{(k)}(x)-u^{(k)}(b))u^{(k+3)}(x)} {2u^{(k+1)}(x)u^{(k+2)}(x)} \\
&=\frac 1 2 + \limsup_{x\to b^-}  \frac 1 2  \frac{(u^{(k)}(x)-u^{(k)}(b))u^{(k+2)}(x)}{u^{(k+1)}(x)^2} \frac{u^{(k+1)}(x) u^{(k+3)}(x)}{u^{(k+2)}(x)^2} \\
&\le\frac 1 2 + \frac 1 2 \; R_{u^{(k)}}(b)\; \limsup_{x\to b^-}\frac{u^{(k+1)}(x) 
u^{(k+3)}(x)}{u^{(k+2)}(x)^2} \\
& = \frac 1 2 \left(1+ R_{u^{(k)}}(b) R_{u^{(k+1)}}(b)\right).
\end{align*}
The inequalities follow from the same arguments as in Step 3.

{\bf Step 8:} Consider $n>2$. Then for $k= 1,\ldots,n-2$ and a positive integer $m$ satisfying $R_{u^{(k)}}(b) \leq \frac{m-1}{m}$, we have that $R_{u^{(k-1)}}(b) \leq \frac{m}{m+1}$. 

Observe that  
\begin{align*}
     & R_{u^{(k-1)}}(b) \le \frac 1 2 \left(1+ R_{u^{(k-1)}}(b) R_{u^{(k)}}(b)\right) \leq \frac 1 2 \left(1+ R_{u^{(k-1)}}(b) \frac{m-1}{m} \right).
     \end{align*}
The first inequality comes from Step 7. The second inequality holds due to Step 6 and that $R_{u^{(k)}}(b) \leq\frac{m-1}{m}$.

Hence, a simple algebraic manipulation leads to
    \begin{align*}  2m R_{u^{(k-1)}}(b) \leq m +  R_{u^{(k-1)}}(b)(m-1) 
     &  \Longleftrightarrow  R_{u^{(k-1)}}(b) \leq \frac {m}{m+1}, 
     \end{align*}
which proves Step 8.

{\bf Step 9:} We assert that $u\in \mathcal{G}_{n,[a,b]}$.

First, consider the case $n=2$ and suppose for the sake of a contradiction that $u^{(1)}(b)> 0$. The continuity of $u^{(1)}$ implies that a $\delta>0$ exists such that $\lim_{x\to b^-}u^{(1)}(x)^2>\delta$. Combining this inequality with 
$$\lim_{x\to b^-} (u(x)-u(b))u^{(2)}(x)=\underbrace{\lim_{x\to b^-} (u(x)-u(b))}_{0}\underbrace{\lim_{x\to b^-} u^{(2)}(x)}_{u^{(2)}(b)}=0\;,$$
implies that
$$ R(b)=\lim_{x\to b^-} \frac{(u(x) - u(b))u^{(2)}(x)}{u^{(1)}(x)^2} = \frac{\lim_{x\to b^-} (u(x) - u(b))u^{(2)}(x)}{\lim_{x\to b^-} u^{(1)}(x)^2} =0\;.$$
Thus, $u \notin LP_{2,[a,b]}$ which is a contradiction. We conclude that $u\in  \mathcal{G}_{2,[a,b]}$.

Second, consider that $n>2$. We assert that $R_{u^{(n-2)}}(b)>0$. Suppose that not and  $R_{u^{(n-2)}}(b) \leq 0$. Step 6 implies that $R_{u^{(n-2)}}(b) = 0$. Then using Step 8 for $k=n-2$ and $m=1$ we get that $R_{u^{(n-3)}}(b) \leq \frac 1 2$. Reiterating Step 8 $k$ times, we obtain that $R_{u^{(n-2-k)}}(b) \leq \frac k{1+k}$. In particular, $R_{u^{(0)}}(b) \leq \frac{n-2}
{n-1}< \frac {n-1}{n}$. Thus, $u \notin LP_{n,[a,b]}$ which is a contradiction. Therefore, $R_{u^{(n-2)}}(b)>0$.

To conclude the proof, because $R_{u^{(n-2)}}(b)>0$ and
$\lim_{x\to b^-} u^{(n-2)}(x)-u^{(n-2)}(b) = 0$, we obtain that $\lim_{x\to b^-} u^{(n-1)} (x) =0$. Thus, $u^{(n-1)}(b)=0$. We conclude from Step 5 that $u\in \mathcal{G}_{n,[a,b]}$. 
\end{proof}

The next Lemma shows that we can approximate a convex and decreasing function by infinitely differentiable, convex and decreasing function. This result is quite standard and is used to show that some stochastic orders have a smooth generator \citep{denuit2002smooth}. Because our result approximation requires an additional condition that relates to the bounded domain we consider in this paper, we provide the proof in the Appendix.

\begin{lemma}\label{lemm:aux}
		Consider a convex and decreasing function $u:[a,b]\to \R$ such that $u^{(1)}(a)$ exists and is finite. Then there is a sequence of infinitely differentiable, decreasing and convex functions $u_n:[a,b]\to \R$ such that $u_n$ converges uniformly to $u$. 
\end{lemma}

The next Lemma completes the proof of the Theorem. We will postpone the proof to the Appendix. 

\begin{lemma} \label{Lemma: LPM iff G} We have $G \succeq_{n+1,[a,b]-BSD} F$ if and only if  $$ \int _{a}^{b}u(x)d(F-G)(x) \leq 0 \text { for all }  u \in \mathcal{G}_{n,[a,b]}.$$
\end{lemma}


	\section{Summary}
	
	In this paper, we provide a generator for the $n$-th degree bounded stochastic dominance orders. We identify sets of risk-averse decision-makers that satisfy curvature conditions, such as a bounded-from-below Arrow-Pratt measure of risk aversion or $n$-convexity, which generate the $n$-th bounded stochastic dominance orders. This result fills a theoretical gap in the literature on stochastic orders by providing a decision-theoretic interpretation for  bounded stochastic orders. Additionally, our findings shed light on  potential limitations of using bounded stochastic dominance tied to the peculiar behavior of decision makers that belong to the generators of these orders. We also show that our results also imply novel inequalities for $n$-convex functions, such as Jensen-type inequalities, non-trivial rankings over lotteries for decision-makers with a bounded-from-below Arrow-Pratt measure of risk aversion, and novel comparative statics results in consumption-savings problems for decision-makers with a globally bounded-from-below prudence measure.



\bibliographystyle{ecta}
\bibliography{LPM}

\section{Appendix}\label{Sec:Appendix}

\begin{proof} [Proof of Proposition \ref{Prop:notequal}]
Without loss of generality assume that $a=0$.   Consider the function $$g(x)= -(b-x)^{n+1}(1- \gamma (b-x))$$
where $\gamma$ will be determined later in the proof. 
Note that
\begin{align*}
    g^{(1)}(x)&= (b-x)^n\left(n+1 - \gamma (n+2) (b-x)\right) \\
    g^{(2)}(x) &= -(n+1)(b-x)^{n-1} \left(n -  \gamma (n+2)  (b-x)\right)\\
    \psi_g(x)& := (n-1)g^{(1)}(x) + g^{(2)}(x)(b-x) = -(b-x)^n \left(n+1 - 2\gamma (n+2)(b-x)\right). 
    \end{align*}
Simple inspection implies that $g^{(1)}(x)>0$, $g^{(2)}(x)<0$ and $\psi_g(x)\leq 0$ for every $x\in [0,b]$ if and only if $g^{(1)}(0)>0$, $g^{(2)}(0)<0$ and $\psi_g(0)\leq 0$. Thus, $\gamma =\frac{n+1}{2b(n+2)}$ makes the three inequalities to hold. Therefore, $g\in AP_{n,[0,b]}$ 
To conclude, we show that $g \notin LP_{n,[0,b]}$. To this extent, we use the characterization of $n,[0,b]$-concave differentiable functions and show that for $x$ close to zero we have
$$ R(x)=\frac {(g(x)-g(b))g^{(2)}(x)}{g^{(1)}(x)^2}<\frac{n-1}n.$$
Noting that $g(b) =0$, we have
\begin{align*}
    R(x)& = \frac{(n+1) (1-\gamma (b-x)) \left(n -  \gamma (n+2) (b-x)\right)}{\left(n+1 - \gamma (n+2) (b-x)\right)^2}.
\end{align*}
Taking $x=0$ and using $\gamma = \frac{n+1}{2b(n+2)}$, we obtain 
$ R(0) = \frac {(n+3)(n-1)}{(n+2)(n+1)}.$ Because for every $n\ge 2$, $\frac{n+3}{(n+2)(n+1)}< \frac 1 n$, we conclude that $R(0)< \frac {n-1}n$. The continuity of $g$ further implies that for $x$ close to zero, $R(x)<\frac {n-1}n$. Thus, $g\notin LP_{n,[0,b]}$.
\end{proof}

\begin{proof} [Proof of Lemma \ref{Lemma:Taylor}]
     Suppose that $u \in U_{n+1,[a,b]}$. We have $g^{(j)}(b) = u^{(j)}(b) - u^{(j)}(b) = 0 $ for each $j=1,\ldots,n$. Suppose that $n$ is an even number. We have $ g^{(n+1)}(x) = u^{(n+1)}(x) \geq 0$, and $g^{(n)}(x) = u^{(n)}(x) - u^{n}(b) \leq 0$ because $u^{(n)}$ is an increasing function. Combining this with the fact that $g^{n-1}(b) = 0$, implies that $g^{n-1}(x) \geq 0$ for all $x \in [a,b]$. Continuing inductively, we conclude that $g \in U_{n+1,[a,b]} $. Hence $g \in \mathcal{G}_{n+1,[a,b]} = U_{n+1,[a,b]} \cap AP _{n+1,[a,b]} = U_{n+1,[a,b]} \cap LP _{n+1,[a,b]}$. The proof follows in a similar manner for the case that $n$ is an odd number. 
\end{proof}

\begin{proof}[Proof of Theorem \ref{Thm:Extended_BSD_n23}]
Recall that for any distribution function $H$ on $[a,c]$, define $H^{1}(x):=H(x)$ and, recursively, $H^{j+1}(x):=\int_a^x H^j(z)dz$ for every positive integer $j$. For $H=F,G$, $j=1,2,3$, and $t\in[a,c]$, the integration-by-parts identity used in Lemma \ref{Lemma: LPM iff G} gives
\begin{equation} \label{Eq:LPM_repeated_identity_extended}
H^{j+1}(t)=\frac{1}{j!}LPM_{j,t}(H).
\end{equation}

Since $F$ and $G$ are supported on $[a,b]$, they assign zero probability mass to $(b,c]$. Hence, for every $u:[a,c]\to\mathbb{R}$,
\begin{equation*}
\int_a^c u(x)dF(x)=\int_a^b u(x)dF(x)
\quad \text{and} \quad
\int_a^c u(x)dG(x)=\int_a^b u(x)dG(x).
\end{equation*}
Moreover, $F(x)-G(x)=0$ for all $x\in[b,c]$.

Fix $n\in\{2,3\}$. By Theorem \ref{Main:Thm}, applied to the interval $[a,c]$, we have 
$U_{n,[a,c]}\cap LP_{n,[a,c]}=U_{n,[a,c]}\cap AP_{n,[a,c]}$. Thus, the utility inequality for every 
$u\in U_{n,[a,c]}\cap LP_{n,[a,c]}=U_{n,[a,c]}\cap AP_{n,[a,c]}$ is equivalent to $G\succeq_{n+1,[a,c]-BSD}F$. By \eqref{Eq:LPM_repeated_identity_extended}, this is equivalent to
\begin{equation} \label{Eq:Extended_BSD_integral_condition}
F^{n+1}(x)-G^{n+1}(x)\geq 0
\quad \text{for all } x\in[a,c].
\end{equation}
It remains to rewrite \eqref{Eq:Extended_BSD_integral_condition} over the interval $[b,c]$ and then translate the resulting conditions into LPM notation.

For $x\in[b,c]$, the recursive definition and the fact that $F(t)-G(t)=0$ on $[b,c]$ imply
\begin{equation*}
F^{2}(x)-G^{2}(x)=F^{2}(b)-G^{2}(b).
\end{equation*}
Therefore,
\begin{equation*}
F^{3}(x)-G^{3}(x)
=
F^{3}(b)-G^{3}(b)
+(x-b)\left(F^{2}(b)-G^{2}(b)\right),
\end{equation*}
and
\begin{align*}
F^{4}(x)-G^{4}(x)
&=
F^{4}(b)-G^{4}(b)
+(x-b)\left(F^{3}(b)-G^{3}(b)\right)  \\
&\quad
+\frac{(x-b)^2}{2}\left(F^{2}(b)-G^{2}(b)\right).
\end{align*}

Consider first $n=2$. Condition \eqref{Eq:Extended_BSD_integral_condition} requires $F^{3}(x)-G^{3}(x)\geq 0$ for all $x\in[a,c]$. By \eqref{Eq:LPM_repeated_identity_extended}, this is equivalent on $[a,b]$ to
\begin{equation*}
LPM_{2,t}(F)\geq LPM_{2,t}(G)
\quad \text{for all } t\in[a,b].
\end{equation*}
On $[b,c]$, the expression
\begin{equation*}
F^{3}(b)-G^{3}(b)
+(x-b)\left(F^{2}(b)-G^{2}(b)\right)
\end{equation*}
is affine in $x-b$, so its minimum on $[b,c]$ is attained at $x=b$ or $x=c$. The condition at $x=b$ is already included in the inequalities on $[a,b]$. The condition at $x=c$ is $F^{3}(c)-G^{3}(c)\geq 0$, which is equivalent by \eqref{Eq:LPM_repeated_identity_extended} to
\begin{equation*}
LPM_{2,c}(F)\geq LPM_{2,c}(G).
\end{equation*}
Thus, the stated LPM conditions are necessary and sufficient when $n=2$.

Now consider $n=3$. Condition \eqref{Eq:Extended_BSD_integral_condition} requires $F^{4}(x)-G^{4}(x)\geq 0$ for all $x\in[a,c]$. By \eqref{Eq:LPM_repeated_identity_extended}, this is equivalent on $[a,b]$ to
\begin{equation*}
LPM_{3,t}(F)\geq LPM_{3,t}(G)
\quad \text{for all } t\in[a,b].
\end{equation*}
On $[b,c]$, the expression
\begin{align*}
F^{4}(b)-G^{4}(b)
&+(x-b)\left(F^{3}(b)-G^{3}(b)\right)  \\
&+\frac{(x-b)^2}{2}\left(F^{2}(b)-G^{2}(b)\right)
\end{align*}
is a quadratic polynomial in $x-b$. Its value at $x=b$ is already covered by the inequalities on $[a,b]$. Its value at $x=c$ is $F^{4}(c)-G^{4}(c)$, which is equivalent by \eqref{Eq:LPM_repeated_identity_extended} to
\begin{equation*}
LPM_{3,c}(F)\geq LPM_{3,c}(G).
\end{equation*}

If $F^{2}(b)-G^{2}(b)\leq 0$, then this quadratic is concave or affine, so its minimum on $[b,c]$ is attained at an endpoint. Hence the endpoint conditions are necessary and sufficient in this case. If $F^{2}(b)-G^{2}(b)>0$, then the quadratic is strictly convex. Its vertex lies in the open interval $(b,c)$ exactly when
\begin{equation*}
-(c-b)\left(F^{2}(b)-G^{2}(b)\right)
<
F^{3}(b)-G^{3}(b)
<
0.
\end{equation*}
When this occurs, non-negativity on $[b,c]$ also requires the value at the vertex to be non-negative. Evaluating the quadratic at the vertex gives
\begin{equation} \label{Eq:vertex_repeated_integral}
F^{4}(b)-G^{4}(b)
-\frac{\left(F^{3}(b)-G^{3}(b)\right)^2}
{2\left(F^{2}(b)-G^{2}(b)\right)}
\geq 0.
\end{equation}
If the vertex does not lie in $(b,c)$, the minimum on $[b,c]$ is attained at an endpoint, so no additional condition is needed.

It remains only to translate the vertex condition into LPM notation. By \eqref{Eq:LPM_repeated_identity_extended},
\begin{equation*}
F^{2}(b)-G^{2}(b)=LPM_{1,b}(F)-LPM_{1,b}(G),
\end{equation*}
\begin{equation*}
F^{3}(b)-G^{3}(b)
=
\frac{1}{2}\left(LPM_{2,b}(F)-LPM_{2,b}(G)\right),
\end{equation*}
and
\begin{equation*}
F^{4}(b)-G^{4}(b)
=
\frac{1}{6}\left(LPM_{3,b}(F)-LPM_{3,b}(G)\right).
\end{equation*}
Therefore, the condition that the vertex lies in $(b,c)$ becomes exactly
\begin{equation*}
LPM_{1,b}(F)-LPM_{1,b}(G)>0
\end{equation*}
and
\begin{equation*}
-2(c-b)\left(LPM_{1,b}(F)-LPM_{1,b}(G)\right)
<LPM_{2,b}(F)-LPM_{2,b}(G)<0.
\end{equation*}
Under this condition, inequality \eqref{Eq:vertex_repeated_integral} is equivalent, after multiplying by 
$24\left(LPM_{1,b}(F)-LPM_{1,b}(G)\right)>0$, to
\begin{equation*}
4\left(LPM_{1,b}(F)-LPM_{1,b}(G)\right)
\left(LPM_{3,b}(F)-LPM_{3,b}(G)\right)
\geq
3\left(LPM_{2,b}(F)-LPM_{2,b}(G)\right)^2.
\end{equation*}
Thus, the stated LPM conditions are necessary and sufficient when $n=3$.

Combining these conditions with the initial equivalence from Theorem \ref{Main:Thm} proves the result.
\end{proof}

\begin{proof}[Proof of Proposition~\ref{prop-consumption}]
Let $n \geq 2$ be an integer. Define the function $k_{s} :[0 ,\overline{y}] \rightarrow \mathbb{R}$ by  $k_{s}(y) : =-u^{ (1) }(Rs +y)$ for all $0 \leq s \leq x$. First note that $k_{s}(y)$ belongs to $U_{n,[0,Rx +\overline{y} -Rs]} \cap AP_{n,[0 ,Rx +\overline{y} -Rs]}$ as $k_{s}^{(j)}$ has the same sign as $-u^{(j+1)}$ and $$k_{s}^{(j)} ( Rx +\overline{y} -Rs ) = -u^{(j+1)}(Rs+ Rx +\overline{y} -Rs) = -u^{(j+1)}(Rx +\overline{y}) =0 $$
for $j=1,\ldots,n-1$ where in the last inequality we used the fact that $U_{n,[0,Rx +\overline{y} -Rs]} \cap AP_{n,[0 ,Rx +\overline{y} -Rs]} = \mathcal{G}_{n,[0,Rx +\overline{y} -Rs]}$ (see Theorem \ref{Main:Thm}). 

Clearly $G  \succeq _{n+1 ,[0 ,Rx +\overline{y}]-BSD} F$ implies that $G \succeq _{n+1 ,[0 ,Rx -Rs +\overline{y}] -BSD}F$ for all $s \in [0 ,x]$.  

Let $w_{s} (s ,q)$ be the derivative of $w$ with respect to $s$. Let $s \in [0 ,x]$. We have 
\begin{align*}w_{s}(s ,G) &  = -u^{ \prime }\left (x -s\right ) +R\int _{0}^{\overline{y}}u^{ \prime }(Rs +y)dG(y) \\
 &  = -u^{ \prime }(x -s) -R\int _{0}^{\overline{y}}k_{s}\left (y\right )dG(y) \\
 &  \leq  -u^{ \prime }\left (x -s\right ) - R\int _{0}^{\overline{y}}k_{s}(y)dF(y) =w_{s}(s ,F) ,\end{align*}
 where the inequality follows from the fact that $G \succeq _{n+1 ,[0 ,Rx -Rs +\overline{y}] -BSD}F$, $k_{s}(y) \in U_{n,[0,Rx +\overline{y} -Rs]} \cap AP_{n,[0 ,Rx +\overline{y} -Rs]}$ and applying Theorem \ref{Main:Thm}. Theorem 6.1 in \cite{topkis1978}
implies that $\rho(F) \geq \rho(G)$.           
\end{proof}

\begin{proof}[Proof of Corollary \ref{Cor:convex}]
Note that $k_{n}(x):=-f(b-x^{1/n})$ is convex on $[0,(b-a)^{n}]$ if and only if $k_{n}^{(2)}(x) \geq 0$ on $[0,(b-a)^{n}]$, i.e., if and only if $x^{1/n}f^{(2)} (b- x^{1/n}) + (n-1)f^{(1)} (b- x^{1/n}) \leq 0$ for all $x \in [0,(b-a)^{n}]$. Letting $y=b-x^{1/n}$ the last inequality is equivalent to $(b-y) f^{(2)} (y) + (n-1)f^{(1)} ( y) \leq 0$ for all $y \in [a,b]$. This implies that $f$ belongs to $AP_{n,[a,b]}$, and hence, Theorem \ref{Main:Thm} implies that $f$ is $n,[a,b]$-concave, i.e., $(f(b)-f(x))^{1/n}$  is convex on $[a,b]$. Similarly, if $(f(b)-f(x))^{1/n}$  is convex on $[a,b]$, then Theorem \ref{Main:Thm} implies that $f$ belongs to $AP_{n,[a,b]}$ which implies that $(b-y) f^{(2)} (y) + (n-1)f^{(1)} ( y) \leq 0$ for all $y \in [a,b]$. Thus, $-f(b-x^{1/n})$ is convex on $[0,(b-a)^{n}]$. 
\end{proof}

\begin{proof}[Proof of Corollary \ref{Corr:jensen}]
Let $X$ be a random variable on $[a,b]$. From Corollary \ref{Cor:convex} the function $-f(b-x^{1/n})$ is convex on $[0,(b-a)^{n}]$. Hence, from Jensen's inequality we have $$  \mathbb{E}f(b-Y^{1/n}) \leq f\left (b-\mathbb{E}(Y)^{1/n} \right ).
$$ 
for the random variable $Y=(b-X)^{n}$. Thus, 
$$ \mathbb{E}f(X) \leq f \left (b -(\mathbb{E}(b-X)^{n})^{1/n} \right ).$$
For the second inequality note that $f$ is increasing since $f \in U_{n,[a,b]}$ and use the monotonicity of the $L^{p}$ norm  to conclude that $(\mathbb{E}(b-X)^{n})^{1/n} \geq \mathbb{E}(b-X)$.  
\end{proof}

\begin{proof}[Proof of Lemma \ref{lemm:aux}]
	We first extend $u$ to be defined on all $\R$. Define $\hat u:\R\to \R$ by
	$$\hat u(x)=\begin{cases}
	u(a)+u^{(1)}(a)(x-a) &\mbox{ if } x <a \\
	u(x) &\mbox{ if } x\in[a,b]\\
	u(b) &\mbox{ if } x>b
	\end{cases}. $$
	By construction the restriction of the function $\hat u$  to $[a,b]$, $\hat u|_{[a,b]}$, is the function $u$. We claim that $\hat u$ is decreasing and convex.
	\begin{itemize}
		\item Monotonicity:  Because $u^{(1)}(a)\le0$, we have that $\hat u$ is decreasing  for $x<a$. For $x\in [a,b]$, $\hat u$ is decreasing because $u$ is decreasing. For $x>b$, the function is constant. Because $\lim_{x\to a+}\hat u(x)=u(a)$ and $\lim_{x\to b^-} \hat u(x)=u(b)$, we conclude that $\hat u$ is decreasing on $\R$.
		\item Convexity: Take $x_1,x_2\in \R$, with $x_1<x_2$, and $\lambda \in (0,1)$. 
		
	For $\lambda x_1+(1-\lambda)x_2\ge b$, by the monotonicity of $\hat u$, we have that $$\hat u (\lambda x_1+(1-\lambda)x_2)=u(b)=\min_{x\in \R} \hat u(x).$$ Hence, $\hat u (\lambda x_1+(1-\lambda)x_2) \le \lambda \hat u ( x_1) +(1-\lambda) \hat u(x_2)$.
	
For $\lambda x_1+(1-\lambda)x_2\in [a,b]$, we split to  three subcases. If $x_1,x_2 \in [a,b]$, we have $\hat u(\lambda x_1+(1-\lambda)x_2) \le \lambda \hat u ( x_1) +(1-\lambda) \hat u(x_2)$ because $u$ is convex. 

If $x_1<a$ and $x_2\in[a,b]$, then $\lambda x_1+(1-\lambda)x_2 = \hat \lambda a+(1-\hat \lambda)x_2$, with $\hat \lambda =\lambda \frac{x_2-x_1}{x_2-a}$.\footnote{Because $\lambda x_1+(1-\lambda)x_2\in [a,x_{2}]$, we have that $\hat \lambda \in [0,1]$.}  
 The convexity of $u$ implies  that\footnote{Because $u$ is convex and differentiable at $a$, we have $u^{(1)}(a)(x-a)+u(a)\le u(x)$ for every $x \in [a,b]$.}
\begin{align*}
& \hat u (x_2)-\hat u(a)\ge u^{(1)}(a) (x_2-a)\\  
     \Longleftrightarrow & \underbrace{(x_1-a)}_{(x_1-x_2)+(x_2-a)} (\hat u (x_2)-\hat u(a)) \le (x_1-a)\quad  u^{(1)}(a) (x_2-a) \\
   \Longleftrightarrow & (x_1-x_2)  (\hat u (x_2)-\hat u(a)) \le (x_2-a) (\underbrace{u^{(1)}(a)(x_1-a) + \hat u(a)}_{\hat u(x_1)} - \hat u (x_2))\\
    \Longleftrightarrow & \frac{x_2-x_1}{x_2-a} (\hat u (a)-\hat u(x_2))\le \hat u(x_1) - \hat u (x_2) \qquad \\
  \Longleftrightarrow &  \underbrace{\lambda\; \frac{x_2-x_1}{x_2-a}}_{\hat \lambda} \;(\hat u (a)-\hat u(x_2)) \le \lambda \; (\hat u(x_1) - \hat u (x_2))\;\\
   \Longleftrightarrow & \hat \lambda \hat u(a)+(1- \hat \lambda) \hat u(x_2) \le  \lambda \hat u(x_1)+(1- \lambda) \hat u(x_2)\;.
\end{align*}
Using that $\hat u$ is convex over $[a,b]$, we obtain
\begin{align*}
\hat u(\lambda x_1+(1-\lambda)x_2 )&= \hat u(\hat \lambda a+(1-\hat \lambda)x_2)\le \hat \lambda \hat u(a)+(1-\hat \lambda) \hat u(x_2)\end{align*}

which proves subcase (ii). 

Finally, if $x_2>b$ we have
		$$\hat u(\lambda x_1+(1-\lambda)x_2) \le \hat u (\lambda x_1+(1-\lambda)b) \le  \lambda \hat u ( x_1) +(1-\lambda) \hat u( b) =  \lambda \hat u ( x_1) +(1-\lambda) \hat u( x_2)$$
		
which proves subcase (iii). The first inequality follows because $\hat u$ is decreasing. The second inequality follows from subcases (i) and (ii).  		
		
For, $\lambda x_1+(1-\lambda)x_2 <a$ we claim that
$$\hat u(\lambda x_1+(1-\lambda)x_2)=\lambda \hat u(x_1) +(1-\lambda)(u^{(1)}(a)(x_2-a)+u(a)) \le \lambda \hat u(x_1) +(1-\lambda) \hat u(x_2) \;.$$ 

Thus, we need to show that $u^{(1)}(a)(x_2-a)+u(a)\le \hat u(x_2)$. If $x_2<a$, the inequality holds from the definition of $\hat u$. If $x_2\in [a,b]$ the inequality holds from the convexity of $u$. If $x_2>b$ we have that $$u^{(1)}(a)(x_{2}-a)+u(a) \leq u^{(1)}(a)(b-a)+u(a) \leq \hat u(b)=\hat u(x_{2}).$$  
The first inequality follows because $u^{(1)}(a) \leq 0$. The second inequality follows from the convexity of $u$.

	\end{itemize}
We prove that $\hat u$ is decreasing and convex. Because a convex function is continuous on the interior of the domain, we have that $\hat u$ is a continuous function.

The next step is based on a mollification argument (see Appendix C in \cite{evans10}) and is quite standard \citep{denuit2002smooth}. Consider a positive mollifier $g$, with a compact support and define $g_n(x)=ng(nx)$. Then, $$\hat u_n(x)=\hat u * g_n (x)=\int \hat u(x-y)g_n(y)dy$$
is infinitely differentiable. Since $\hat u$ is continuous, we have that $\hat u_n$ converges uniformly to $\hat u$ on compact subsets of $\R$ (see Appendix C, Theorem 6, in \cite{evans10}). In particular, we have that $\hat u_n$ converges uniformly to $ \hat u =u $ on $[a,b]$.

We assert that $\hat u_n$ is convex, decreasing, which implies that $u_n:=\hat u _n|_{[a,b]}$ is convex decreasing, infinitely differentiable, such that $u_n$ converges uniformly to $u$.

Indeed, take $x_1<x_2$ and $\lambda\in (0,1)$, then:
\begin{itemize}
	\item Monotonicity: Because $\hat u$ is decreasing and $g\ge 0$, we have that for every $y$, $$\hat u(x_1-y)g_n(y) \ge \hat u(x_2-y)g_n(y)\;.$$ Hence, integrating over $y$, we obtain that $\hat u_n(x_1)\ge \hat u_n(x_2)$.
	\item Convexity: Because $\hat u$ is convex and $g\ge 0$, we have that for every $y$, $$\hat u(\lambda x_1 + (1-\lambda) x_2-y)g_n(y) = \hat u(\lambda (x_1-y) +(1-\lambda)(x_2-y))g_n(y) \le \lambda \hat u(x_1-y)g_n(y) +(1-\lambda)\hat u(x_2-y)g_n(y)\;.$$ Hence, integrating over $y$, we obtain that $\hat u_n(\lambda x_1 +(1-\lambda )x_2)\le \lambda \hat u_n(x_1) +(1-\lambda)\hat u_n(x_2)$.
\end{itemize}
\end{proof}

\begin{proof}[Proof of Lemma \ref{Lemma: LPM iff G}] Assume that $G \succeq_{n+1,[a,b]-BSD} F$. For every positive integer $j$ and a distribution function $W$ on $[a,b]$ define recursively $W^{j+1}(x) := \int _{a} ^{x} W^{j}(z) dz $ for all $x \in [a,b]$ where $W_{1}:=W$. Integration by parts yields 
\begin{equation} \label{Eq:Fish}
    W^{j+1}(c) = \frac{1}{j!}\int_{a}^{c} (c-x)^{j} dW(x) =  \frac{1}{j!}\int_{a}^{b} M_{j,c} (x) dW(x) \end{equation} 
for every positive integer $j$ where  $M_{j,c} (x) = \max \{c-x,0\}^{j}$. 


Let  $u \in \mathcal{G}_{n,[a,b]}$. Using integration by parts for a Lebesgue-Stieltjes integral  multiple times yield
\begin{align*}\int _{a}^{b}u(x)d(F-G)(x) &  =u(x)\left (F(x) -G(x)\right ) \Big|_{a^-}^{b^+} -\int _{a}^{b}u^{ (1) }(x)\left (F(x) -G(x)\right )dx \\
 &  =  -u^{ (1) }(x)\int _{a}^{x}\left (F(z) -G(z)\right )dz \Big |_{a}^{b} +\int _{a}^{b}u^{(2)} (x)\int _{a}^{x}\left (F(z) -G(z)\right )dzdx \\
 &  = u^{(2)} (x)\int _{a}^{x}\int _{a}^{y}\left (F(z) -G(z)\right )dzdy \Big |_{a}^{b} -\int _{a}^{b}u^{ (3) }(x)\int _{a}^{x}\int _{a}^{y}\left (F(z) -G(z)\right )dzdydx \\
 &  =u^{(2)} (b)\int _{a}^{b}\int _{a}^{y}\left (F(z) -G(z)\right )dzdy -\int _{a}^{b}u^{ (3) }(x)\int _{a}^{x}\int _{a}^{y}\left (F(z) -G(z)\right )dzdydx \\ 
 &   =u^{(2)} (b)(F^{3}(b) - G^{3}(b) ) -\int _{a}^{b}u^{ (3) }(x) (F^{3}(x) - G^{3}(x) ) dx. \end{align*}
In the second equality we use the fact that $F(b^+) -G(b^+) =F(a^-) -G(a^-) =0$. In the third equality we use the fact that $u \in \mathcal{G}_{n,[a,b]}$ implies that $u^{ (1) }(b) =0$. Continuing integrating by parts and using the fact that $u \in \mathcal{G}_{n,[a,b]}$ implies that  $u^{ (k) }(b) =0$ for all $k=1,\ldots,n-1$ yield 
\begin{align*}
     \int _{a}^{b}u(x)d(F-G)(x) & =  (-1) ^{n} u^{(n)}(b) (F^{n+1}(b) - G^{n+1}(b) ) + (-1)^{n+1} \int _{a}^{b} u^{(n+1)} (x) (F^{n+1}(x)-G^{n+1}(x) ) dx \\
     & = \frac{1} {n!}(-1) ^{n} u^{(n)} (b) \left ( \int _{a} ^{b} M_{n,b}(x) d (F- G) (x)  \right ) \\ 
     & + \frac{1} {n!} \int _{a} ^{b} (-1)^{n+1} u^{(n+1)} (x) \left (\int _{a} ^{b} M_{n,x}(z) d (F- G) (z)  \right ) dx \leq 0.
     \end{align*}
The equality follows from Equation (\ref{Eq:Fish}). 
The inequality follows from the facts that $(-1)^{k} u^{(k)}  \leq 0$ for $k=n,n+1$ and $G \succeq_{n+1,[a,b]-BSD} F$.

Now assume that $\int_a^b v(x) dG(x)\ge \int_a^b v(x)dF(x)$ for every function $v \in \mathcal{G}_{n,[a,b]}$. In what follows, we show that for every lower partial moment function  $u(x)=\max\{c-x,0\}^n$, there exist a sequence of functions $(u_{j})$ such that $u_j\in  \mathcal{G}_{n,[a,b]}$ for every integer $j$ and $u_j$ converges to $-u$ weakly, i.e., $\lim _{j \rightarrow \infty} \int_a^b u_j(x)dW(x) = -\int_a^b u (x) dW(x)$ for every distribution function $W$ on $[a,b]$. Therefore, proving that $G \succeq_{n+1,[a,b]-BSD} F$.

Let $u(x)=\max\{c-x,0\}^n$ be a lower partial moment function. First, notice that the function $-u$ is $n-1$ times differentiable with $-u^{(n-1)}(x)=n!(-1)^{n} \max\{c-x,0\}$. Hence,  $-u^{(n-1)}$ is either a concave and increasing function or a convex and decreasing function. For simplicity of the proof, assume that $n$ is even, i.e., $-u^{(n-1)}$ is convex and decreasing (the proof for the concave and increasing case is the same). Because $-u^{(n-1)}(a)$ exists, Lemma~\ref{lemm:aux} implies the existence of a sequence $(g_j)$ of decreasing, convex, and infinitely differentiable functions such that $g_{j}$ converges uniformly to $-u^{(n-1)}$.\footnote{The Lemma states the result for $f$ that is convex, decreasing, such that $f^{(1)}(a)$ exists. Clearly, it also holds for the case that $f$ is a concave and increasing function such that $f^{(1)}(a)$ exists.} 

Define recursively $G_{j,k}(x) = \int _{a}^{x} ( G_{j,k-1} (z) -G_{j,k-1}(b) ) dz$ for all $x \in [a,b]$ and every positive integer $k$ where $G_{j,1}(x):=g_{j}(x)-g_j(b)$. Let $u_j(x):=G_{j,n}(x)$, in the next two steps we show that (i) $ u_j\in  \mathcal{G}_{n,[a,b]}$ and  (ii) $u_j$ converges to $-u$ weakly. 

{\bf Step 1.} Notice that because the function $g_j$  is infinitely differentiable we have $u_j \in C^{n+1}([a,b])$. In addition, $u_{j}$ satisfies 
$$u_j^{(k)}(x) =\begin{cases}
G_{j,n-k}(x) - G_{j,n-k}(b) &\mbox{ if }k=1,\ldots,n-1 \\
g_j^{(k+1-n)}(x) &\mbox{ if } k=n,n+1.
\end{cases}
$$
Hence, $u_j^{(k)}(b) =0$ for $k=1,\ldots,n-1$. It remains to show that $u_j\in U_{n,[a,b]}$. Because $g_j$ is convex and decreasing and $n$ is an even number, we have that $(-1)^k u_j^{(k)}(x) \le 0$ for $k=n,n+1$. For $k\le n-1$, we show by induction that the function $G_{j,k}$ is decreasing if $k$ is odd and increasing if $k$ is even. The base case $k=1$ is straightforward since $G_{j,1}(x)=g_{j}(x)-g_j(b)$ and $g_j$ is a decreasing function. For the induction step, notice that $G^{(1)}_{j,k}(x) = G_{j,k-1}(x)-G_{j,k-1}(b)$. Therefore if $G_{j,k-1}$ is a decreasing (increasing) function we have that $G^{(1)}_{j,k}(x)$ is positive (negative) and, hence, $G_{j,k}$ is an increasing (decreasing) function. We conclude that $(-1)^k u_j^{(k)}\le 0$ for every $k=1,\ldots,n+1$. Thus, $u_j\in U_{n,[a,b]}$ and, therefore, $u_j\in \mathcal{G}_{n,[a,b]}$.

{\bf Step 2.} We claim that for any distribution function $F$ on $[a,b]$ we have $\int_a^b u_j(x)dF(x) \to -\int_a^b u(x) dF(x)$. Because $u_j^{(k)}(b)=0$ and  $u^{(k)}(b)=0$ for $k=1,\ldots,n-2$, integration by parts implies that
\begin{align}
    \int_a^b u_j(x) dF(x) &= (-1)^{n-1}\int_a^b u_j^{(n-1)}(x) F^{n-1}(x) dx,  \text {   and } \\ 
     -\int_a^b u(x) dF(x) &= (-1)^{n-1}\int_a^b -u^{(n-1)}(x) F^{n-1}(x) dx.
\end{align}
	 Because $u_j^{(n-1)} (x)= g_j(x)-g_j(b)$, $g_j$ converges uniformly to $-u^{(n-1)}$ and $\lim _{j \rightarrow \infty} g_j(b)=0$, we conclude using the 
	 dominated convergence theorem that $\lim_{j\to \infty}  \int_a^b u_j(x) dF(x) =  -\int_a^b u(x) dF(x) $.
\end{proof}

\end{document}